\documentclass[11pt,reqno]{article}
\usepackage[margin=1.0in]{geometry}
\usepackage[english]{babel}
\usepackage[utf8]{inputenc}
\usepackage{amsmath,amssymb,amsthm, comment, mathtools,ulem,hyperref,mathrsfs,xcolor,enumitem}
\mathtoolsset{showonlyrefs}

\def\ve{\varepsilon}

\newcommand{\eee}{equation}
\newcommand{\be}{\begin{\eee}}
\newcommand{\ee}{\end{\eee}}

\newcommand{\nrm}[1]{\Vert#1\Vert}

\newcommand{\nnrm}[1]{{\vert\kern-0.25ex\vert\kern-0.25ex\vert #1 
		\vert\kern-0.25ex\vert\kern-0.25ex\vert}}

\newcommand{\tht}{\theta}

\newcommand{\gmm}{\gamma}

\newcommand{\dlt}{\delta}

\newcommand{\rmd}{{\rm d}}
\newcommand{\id}{{\rm id}}

\newcommand{\Dm}{{\mathscr{D}_\mu(M)}}

\newtheorem{proposition}{Proposition}[section]
\newtheorem{theorem}[proposition]{Theorem}
\newtheorem{corollary}[proposition]{Corollary}
\newtheorem{lemma}[proposition]{Lemma}
\theoremstyle{definition}
\newtheorem{definition}[proposition]{Definition}
\newtheorem{remark}[proposition]{Remark}

\theoremstyle{definition}

\numberwithin{equation}{section}
\newcounter{nmdthmcnt}

\newenvironment{namedthm2}[2][]{\addtocounter{nmdthmcnt}{1}%
    \theoremstyle{plain}\newtheorem*{nmdthm\roman{nmdthmcnt}}{#2}%
    \begin{nmdthm\roman{nmdthmcnt}}[#1]}{\end{nmdthm\roman{nmdthmcnt}}}

\usepackage{tikz}
\usetikzlibrary{patterns}
\usetikzlibrary{fpu}
\usetikzlibrary{plotmarks}
\usetikzlibrary{decorations.markings}
\usepackage{pgfplots}
\usetikzlibrary{intersections, pgfplots.fillbetween}

\title{Twisting in Hamiltonian Flows and Perfect Fluids}
\begin{document}

\date{\today}
\author{Theodore D. Drivas\footnote{Department of Mathematics, Stony Brook University, Stony Brook, NY, 11794, USA\\
\phantom{ads} Email address: {\tt tdrivas@math.stonybrook.edu}},  \,Tarek M. Elgindi\footnote{Mathematics Department, Duke University, Durham, NC 27708, USA
Email address: {\tt tarek.elgindi@duke.edu}}, \, and In-Jee Jeong\footnote{Department of Mathematical Sciences and RIM, Seoul National University. Email address: {\tt injee\_j@snu.ac.kr}}}
\maketitle

\begin{abstract}We introduce a notion of stability for non-autonomous Hamiltonian flows on two-dimensional annular surfaces. This notion of stability is designed to capture the sustained twisting of particle trajectories. The main Theorem is applied to establish a number of results that reveal a form of irreversibility in the Euler equations governing the motion of an incompressible and inviscid fluid.
In particular, we
show that nearby general stable steady states (i) all fluid flows exhibit indefinite twisting (ii) vorticity generically exhibits gradient growth and  wandering. We also
give examples of infinite time gradient growth for smooth solutions to the SQG equation and
of smooth vortex patches  that entangle and develop unbounded
perimeter in infinite time.
\end{abstract}

\tableofcontents

\section{Introduction}

Stability is one of the cornerstones of classical mechanics and the theory of differential equations. Stable configurations are those that are robust and generally observed. As such, determining which configurations in a given system are stable, and in what sense, is of great importance. When considering the motion of particles in space, such as in a fluid,  there are two notions of stability that one could consider.  \textit{Eulerian} stability pertains to the stability of the velocities of the particles whereas \textit{Lagrangian} stability pertains to the stability of their positions.
A simple example illustrating the difference between these two notions of stability is given by the motion of a single particle in free-space. Let us denote by $x(t)\in\mathbb{R}^d$ and $v(t)\in\mathbb{R}^d$ the location and velocity, respectively, of the particle at time $t\in\mathbb{R}$. Under the assumption of no external force on the particle,  the particle moves with its unchanging velocity:
\[\dot{x}(t)=v(t),\]
\[\dot{v}(t)=0.\]
Therefore, if we know the velocity of the particle at any fixed time $t=t_0$, we can find the location of the particle for all time by integrating the equations of motion. If at $t=t_0$ there was an ever-so slight inaccuracy in our measurement of the velocity of the particle, this would not affect the system too much from the Eulerian point of view (trivially for this system). Indeed, 
\[|v(t_0)-v_*(t_0)|\leq \ve \implies |v(t)-v_*(t)|\leq \ve,\] for all time $t\in\mathbb{R}$ for any $\ve>0$. In particular, slight errors in our measurement of the velocity at some time $t_0$ do not accumulate over time and do not lead to any significant error in our prediction of the velocity at later times. This is an example of \textit{Eulerian stability}. 

As far as the motion of the particle goes, however, making any error in the measurement of the velocity \textit{will} lead to large deviations in the positions of the particle. Indeed, two particles $(x(t),v(t))$ and $(x_*(t),v_*(t))$ leaving from the same location satisfy
\[|v(t_0)-v_*(t_0)|=\ve\implies |x(t)-x_*(t)|=\ve |t-t_0|,\] for all $t\in\mathbb{R}$. In particular, the distance between them  becomes arbitrarily large over time. This example thus exhibits \textit{Lagrangian instability}\footnote{Another interesting example is that of objects orbiting around a fixed center with slightly different angular velocity, maybe due to the effect of gravitation. In this case, the distance between the objects can regularly increase and decrease in time.}. Our purpose in mentioning this example is to say that {Eulerian} stability is the only reasonable form of stability that can be satisfied by the system as far as the strict notion of stability is concerned. As far as stability in the strict sense of closeness of norms goes, Lagrangian stability should not be expected. 

A natural question one can ask is whether we can relax the notion of Lagrangian stability in such a way that we can still deduce useful information about the long-time behavior of particle motion. This relaxation will be done on two fronts: first, by focusing our attention on the behavior of a continuum of (interacting) particles, and second, by restricting our attention to certain qualitative and quantitative Lagrangian features. An important setting for the applications we have in mind is when the motion is also area preserving, or incompressible. Maintaining this constraint defines the interaction between the particles.
Since we will discuss continua of particles, it makes sense to consider the so-called particle trajectory map. On a smooth Riemannian manifold  $M$, we can consider a suitably smooth, divergence-free, and time-dependent velocity field $v:M\times \mathbb{R}\rightarrow TM$. In this case, we can define the associated particle trajectory map:
\begin{align}
    \frac{\rmd }{\rmd t}\Phi(x,t)&=v(\Phi(x,t),t),\\
\Phi(x,0)&=x.
\end{align}
 Note that for each $t\in\mathbb{R}$, $\Phi(\cdot,t)$ is a volume-preserving diffeomorphism of $M$.
The first relaxation of the notion of stability we have in mind, thus, is to consider the stability of the map $\Phi(\cdot, t)$ rather than a single trajectory $\Phi(x,t)$. To give a taste of the type of Lagrangian features we will discuss later, let us mention a version of the stability theorem we will establish. 


\begin{theorem}\label{thmmaindiff}
Let $M= \mathbb{T}\times [0,1],$ the periodic channel. Fix $v_*\in C^1(M)$ defined by $v_*(x_1,x_2)=(V(x_2),0)$, with $V$ strictly monotone. Consider a time-dependent divergence-free velocity field $v\in C^1(M\times \mathbb{R})$ and its flow map $\Phi(\cdot,t)=(\Phi_1(\cdot,t),\Phi_2(\cdot,t))$. Assume that $\|v(\cdot,t)-v_*(\cdot)\|_{L^2(M)}<\ve$ for all $t\in\mathbb{R}$. Then,
\be\label{updown}
\|\Phi_2(x,t)-x_2\|_{L^2}\leq C \sqrt{\ve} \log (1+|t|)
\ee for all $t\in\mathbb{R}$, for some fixed $C$ that depends only on $V$. 
Moreover
\be\label{sideside}
\|\widetilde{\Phi}_1(x,t)- x_1 - v_*(\widetilde{\Phi}_2(x,t)) t\|_{L^2}  \leq C\sqrt{\ve} |t| 
\ee
for all $t\in\mathbb{R}$. In \eqref{sideside}, $\widetilde{\Phi}(\cdot,t)$ is the lift of the flow to the universal cover $\widetilde{M}= \mathbb{R}\times [0,1]$.
\end{theorem}

This result is a special case of Theorem \ref{twistingthm2} below, while Theorem \ref{twistingthm} gives a localized version.

We first comment on the estimate \eqref{updown}.
In the above setting, since $v$ is completely non-autonomous, it is easy to move a single particle vertically a distance of order $1$ in time on the order of $\frac{1}{\ve}$; however, it takes an exponentially longer amount of time to accomplish this task for a mass of particles. This is reminiscent of the work of Nekhoroshev \cite{Nek} on Arnold diffusion \cite{As} that concerns the motion of single trajectories in nearly integrable systems. 
Theorem \ref{thmmaindiff} says that, while horizontal motion may be highly unstable in this setting in the strict sense, vertical motion is much more stable. In fact, the proof of this theorem follows from a stronger result that establishes a type of stability of the motion on the universal cover of $M$. 

Next we comment on the estimate \eqref{sideside}, which we refer to as ``stability of twisting" (twisting is defined in Definition \eqref{DefinitionTwisting}). This is established by a certain rigidity, of a topological nature, of the ``lifted" dynamics of particles on $\mathbb{R}\times [0,1]$, the universal cover of $M$. In that setting, the number of times a particle has wound around the circle in $\mathbb{T}\times [0,1]$ can be described simply by the change in its horizontal location in $\mathbb{R}\times [0,1].$ The average winding of the particles can then be described by integral quantities on the universal cover. The estimation of these integral quantities relies heavily on topological constraints on the image of the particle trajectory map in the universal cover. Indeed, the more wound up the particles become, the more difficult it is to unwind them. It is noteworthy that the proof of the stability of twisting theorem is of a global nature since every particle starting in a tubular neighborhood could, in principle, exit the neighborhood in finite time.   




\begin{remark}[On the Calabi Invariant]
The results and techniques on twisting of flows introduced here should be considered in relation to the celebrated Calabi invariant, which is a measure of the overall winding in a slightly more restricted setting of diffeomorphisms that leave points on the boundary fixed, see e.g. \cite{AK}.  More precisely  if  $\varphi$ is such a diffeomorphism, then this Calabi invariant is computed by averaging the winding of all pairs of points along \textit{any} isotopy from $\varphi$ to the identity \cite{Fathi,Fathi2}.  This number ends up being independent of the isotopy and hence an invariant of the group of said diffeomorphisms.  This object was used in the first proof of the unboundedness of the diameter of the group of area-preserving diffeomorphisms \cite{ER}. To prove our main results, we introduce quantities that appear to be some localized analogue of the Calabi invariant in our context. Though we do not establish any general invariant properties per se, the bounds we prove indicate that there may be some invariance at play (see Lemma \ref{keylemma}). 
\end{remark}

\subsection{On the long-time behavior of 2d ideal fluid flows}

 We now move to discuss applications to the 2d Euler equation. The 2d Euler equation can be thought of as a 2d Hamiltonian flow where the Hamiltonian and the flow are coupled through a non-linear and non-local law. Indeed, the Euler equation on two dimensional domains $M$ can be written as:
\be
\frac{\rmd}{\rmd t}\Phi(x,t)= \nabla^\perp \psi(\Phi(x,t),t), \qquad  \psi(x,t) :=  K_M[\omega_0\circ \Phi^{-1}(\cdot, t)](x)
\ee where $K_M$ is the Green's function of the  Laplacian on $M$ (with appropriate boundary conditions when $\partial M$ is non-empty, see \S 2.2 of \cite{DE}). 
Here, $\omega_0$ is the initial vorticity of the fluid. Thus, for each $\omega_0$, the particle trajectory map $\Phi(\cdot, t)$ solves a 2d Hamiltonian system where the Hamiltonian is determined by the particle trajectory map itself. Understanding the long-time behavior of 2d Euler flows is a major challenge and what has been established rigorously is quite far from heuristic arguments and expectations \cite{DE}. Major progress has been made using mixing, though in somewhat restrictive settings, starting with the breakthrough work \cite{B}. The long-time behavior of 2d Euler flows in the large, however, remains wide open and new tools are needed.  

As an application of our stability results for general 2d Hamiltonian flows,  we establish some qualitative forms of stability for the 2d Euler equations that are Lagrangian in nature. We prove:
\begin{namedthm2}{Results on inviscid fluids}{\rm [Informal Statement]} Open sets of solutions of the two-dimensional Euler equation on annular domains close to  Eulerian stable equilibria which induce some differential rotation (shearing) between particles  exhibit:
\begin{enumerate}
    \item {\bf ``Aging" of the flow (Theorem \ref{agethm}):} The distance between the initial particle configuration and the particle configuration at time $t$ diverges linearly as $|t|\rightarrow\infty$. The amount of energy required to undo the twisting grows linearly in time. 
    \item {\bf Filamentation, spiraling, and wandering (Theorems \ref{gengrowth} and \ref{wanderingthm}):}  Level sets of the vorticity filament and spiral infinitely as $t\rightarrow\infty.$ Nearby neighborhoods wander in $L^\infty.$
    \item {\bf Infinite time blowup for SQG on $\mathbb{T}^2$ (Theorem  \ref{sqgthm}):}  We exhibit smooth solutions of the SQG equation that display unbounded gradient growth.
    \item {\bf Unbounded Perimeter growth for vortex patches (Theorem \ref{thm:patch}):}  We give examples of a smooth vortex patch on $\mathbb{R}^2$ whose perimeter grows at least linearly in time.
\end{enumerate}
\end{namedthm2}

  The mechanism we identify for loss of smoothness and wandering is greatly influenced by the work of Nadirashvili \cite{N}. There, an analogue of point 2 above was established in a more restrictive setting by crucially using the invariance of the boundaries of the fluid domain.  In essence, our work justifies that one can replace Nadirashvili’s fixed annulus by a topological time-dependent annulus constrained by two levels of an appropriate stream function, and this is at the heart of all our results. Points 1 and 2 of our main results validate a conjecture of Yudovich \cite{Y1,Y2} on unbounded growth of solutions to the Euler equation nearby stable steady states. 
Concerning points 3 and 4, no rigorous all-time results were previously available; see, respectively, \cite{HeKiselev, KN} and \cite{CJ, EJSVPII, D, MMZ}.

\section{Stability of Twisting in Hamiltonian Flows}\label{highway}

On a general annular domain\footnote{More generally, we can consider two-dimensional surfaces.} $M\subset \mathbb{R}^2$, consider a Hamiltonian $\psi:M\times \mathbb{R}\to \mathbb{R}$ and its flow 
\[
\frac{\rmd}{\rmd t} \Phi(x,t)= \nabla^\perp \psi(\Phi(x,t),t).
\]
It is supposed that the flows leave the boundaries invariant (the Hamiltonian is constant on connected components of the boundary, making the velocity field tangent).
When $\psi(x,t)=\psi_0(x)$ is autonomous, 
it is well known that the domain is foliated by periodic orbits and orbits connecting fixed points. This follows directly from the proof of the Poincar\'{e}--Bendixson theorem.  Indeed, since the flow preserves the level sets of the Hamiltonian:
\[\frac{\rmd}{\rmd t}\psi_0(\Phi(x,t))=0,\]
the system can be exactly solved using action-angle variables \cite{A2}. In particular, on any annular subdomain of $M$ foliated by periodic orbits, coordinates $(r,\theta)\in A:=[r_0,r_1]\times \mathbb{S}^1$ can be found so that the system becomes exactly
\[\dot\theta(t) = v(r(t)),\qquad \dot{r}(t) =0.\]
The exact solution in these coordinates is then:
\[(r(t),\theta(t))=\Big(r(0),\theta(0)+v(r(0))t\Big).\] 
When $v$ is non-constant, the flow “twists” the particles in the annular region $A$. For example, if $v(r_1)>0>v(r_0)$, we find that particles on the circle $r=r_1$ are constantly rotated counterclockwise, while particles on the circle $r=r_0$ are constantly rotated clockwise.  While this twisting does not lead to the exponential separation of particles, as in truly chaotic dynamics, it does lead to linear separation of particles, which might be measured by the  distortion bound
\[\|\nabla\Phi(\cdot, t)\|_{L^\infty}\geq c t.\]  
The $c>0$ in the above inequality is just a constant that measures the degree of twisting ($v(r_1)-v(r_0)$ in the example above). Another way to measure the twisting induced by autonomous Hamiltonian flows is to study the image of curves that are transversal to the flow lines. For example, we may consider the image of the curve $\theta=0$ in action-angle variables. It is easy to see that this is mapped to the curve 
\[\Gamma_t:=\{(r,v(r)t): r\in [r_0,r_1]\}.\]
The length of $\Gamma_t$ grows linearly in time and spirals around the center of the annulus (in fact, it intersects each line $\theta=\theta_0$ approximately $c  t$ times as $t\rightarrow\infty$). 

To make all of the above statements about twisting, we relied crucially on the fact that $\psi=\psi_0$ was autonomous. In the non-autonomous case, the analysis above breaks down because there need not be any periodic orbits in the system. In fact, it is not even clear how one should define twisting when the Hamiltonian is non-autonomous. As a step in this direction, we establish a general stability  of twisting result. Namely, we show that non-autonomous perturbations of autonomous flows exhibit twisting in the senses we mentioned above (gradient growth and spiral formation). A conceptual difficulty in making this work is that, in the autonomous case, we had to rely on the motion of single particles. In the non-autonomous case, singling out finitely many particles for study is impossible. Indeed, the flow could evolve in time in a conspiring way such that no one particle behaves in a predictable way. To fix this problem, we first uncover a softer notion of ``twisting'' that encompasses the above facts about the flow. Second, we introduce integral quantities that capture this twisting for masses of particles and for which we can prove some stability theorems.

\subsection{Twisting and its stability}

We now turn to define ``twisting." Consider a family of homeomorphisms $\Phi(\cdot, t):M\rightarrow M$ depending continuously on the parameter $t\in \mathbb{R}$. Since $M$ is an annular surface, we could equip it with global polar coordinates $(r,\theta)$ with $r\in [a,b]$ and $\theta\in \mathbb{T}.$ Given any single trajectory, there is a unique way to lift $\Phi(x,t)$ (which lives on $M$) to $\widetilde{\Phi}(x,t)$ living on $[a,b]\times \mathbb{R}$ for which the $\theta$ coordinate of $\widetilde{\Phi}(x,t)$ is continuous in time. 

\begin{definition}\label{DefinitionTwisting}
An incompressible and continuous flow $\Phi(\cdot,t):M\rightarrow M$ is said to be twisting if the diameter of $\widetilde{\Phi}(M,t)$ becomes unbounded as $|t|\rightarrow \infty.$  
\end{definition}

	\begin{figure}[h!]
		\centering
		\includegraphics[]{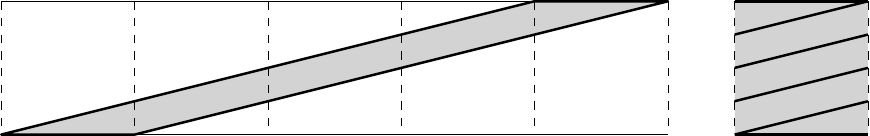}  
		\caption{Evolution of $\widetilde{\Phi}(M,t)$ (colored region) under linear shear} 
  \label{fig:shear}
	\end{figure}

Since there is a direct correspondence between velocity fields and the flows they produce, we may also call the velocity fields twisting if their flows are twisting. A simple example of a twisting flow on $\mathbb{T}\times [0,1]$ generated by the twisting velocity $u=(x_2,0)$ is given by \[\widetilde{\Phi}(x,t)=(x_1+t x_2,x_2),\] while  a non-twisting flow generated by  velocity $u=(x_2\cos(t),0)$ on the same domain is given by 
\[\widetilde{\Phi}(x,t)=(x_1+x_2\sin(t),x_2).\] 

The above definition of twisting generalizes the notion of a ``twist map" in the setting of flows. Indeed, recall that an autonomous Hamiltonian $\psi_*= \psi_*(x)$ generates a \textit{twist map} $\Phi^*(\cdot, t)$ provided that the time a particle takes to orbit the  level set $\{\psi_*=c\}$:
\be\label{traveltime}
\mu(c) := \int_{\{\psi_*=c\}} \frac{\rmd \ell}{|\nabla \psi_*|}
\ee
is non-constant on $\partial M$. It is easy to see that for a twist map the diameter of $\widetilde{\Phi}(M,t)$ grows like a constant multiple of $|t|$ as $|t|\rightarrow\infty.$ Our main theorems establish local and global versions of the stability of twisting. An interesting problem is to study the abundance of twisting flows among all incompressible flows. It is easy to see that twisting velocity fields are dense in class of incompressible velocity fields since the perturbations can be non-autonomous. Theorem \ref{twistingthm} shows that \textit{every} velocity field in an $L^1$ neighborhood of an autonomous twisting velocity field is twisting. 


\subsection*{Stability of Twisting Theorems}

  Under appropriate assumptions about proximity to twist maps, local or global, we can now state our stability theorems. We state the theorems on annular domains in $\mathbb{R}^2$ and extensions to other domains are given in \S \ref{sect:ext}. Henceforth, we fix an annular domain $M\subset\mathbb{R}^2.$
\\

\vspace{-2mm}
\noindent \textbf{Assumption 1.}  An autonomous and smooth stream function $\psi_*$ is said to satisfy Assumption 1 on an open set $\mathsf{A}\subset M$ if it has two level sets contained in $\mathsf{A}$ that are non-contractible and have distinct travel times \eqref{traveltime}. 
\medskip 

To state our theorem, we introduce a distance: For a given smooth $\psi:M\times\mathbb{R}\rightarrow\mathbb{R}$, we define
\[\|\psi-\psi_*\|_{\mathsf{A}}:=\sup_{T\in\mathbb{R}}\Big|\frac{1}{T}\int_0^T\|\psi(t)-\psi_*\|_{L^1(\mathsf{A})}\rmd t+\frac{1}{T}\int_0^T\|\nabla^\perp\psi(t)\cdot\nabla\psi_*\|_{L^1(\mathsf{A})}\rmd t\Big|. \]

\begin{theorem}[Local Stability of Average Twisting]\label{twistingthm}
Assume that $\psi_*$ satisfies {\rm \textbf{Assumption 1}} on $\mathsf{A}\subset M.$ Assume that $\psi$ is smooth on $M$ and $u=\nabla^\perp\psi$ is uniformly bounded (in time). Then, if $\|\psi-\psi_*\|_{\mathsf{A}}$ is sufficiently small depending only on $\psi_*|_{\mathsf{A}}$, the Lagrangian flow $\Phi(\cdot,t)$ associated to $\psi$ satisfies 
\begin{itemize}
    \item For any lift $\widetilde{\Phi}(\cdot, t)$ of the flow  to the universal cover $\widetilde{M}\cong \mathbb{R}\times [0,1]$, the diameter of $\widetilde{\Phi}(M,t)\subset \widetilde{M}$ grows linearly in time.
    \item $\|\nabla\Phi(\cdot, t)\|_{L^1}$ diverges linearly as $|t|\rightarrow\infty.$
\end{itemize}
\end{theorem}


	\begin{figure}[h!]
		\centering
		\includegraphics[]{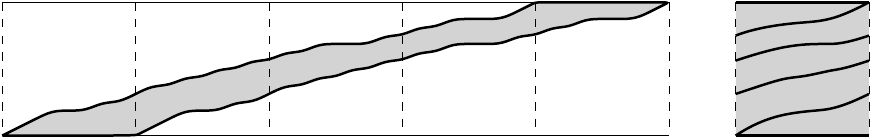}  
		\caption{Evolution of $\widetilde{\Phi}(M,t)$ (colored region) under a perturbed shear} 
  \label{fig:shear2}
	\end{figure}

\begin{remark}
    It is important to emphasize that this result is completely local. Beyond the qualitative assumption of smoothness, nothing is assumed about $\psi_*$ or $\psi$ outside of the region $\mathsf{A}$ (in fact, continuity of the flow map is all that is needed). Only the second statement requires the uniform boundedness of $u=\nabla^\perp \psi$ in space-time. 
\end{remark}

\begin{remark}
While the Theorem applies to the case $M=\mathbb{T}\times [0,1],$ it \textit{does not hold} on the domain $\mathbb{R}\times[0,1]$ itself; periodicity in $x_1$ is crucial.
\end{remark}

We will now state a global stability theorem on twisting for the lifted dynamics.
\medskip 

\noindent {\rm\textbf{Assumption 2.}} An autonomous and smooth stream function $\psi_*$ on $M$ is said to satisfy Assumption 2 if all its level sets are non-contractible loops that foliate $M$.
\medskip 

To state our second theorem, we introduce another distance: 
given $\psi:M\times\mathbb{R}\rightarrow \mathbb{R}$, define
\[\|\psi-\psi_*\|_{2}:=\sup_{T\in\mathbb{R}}\Big|\frac{1}{T}\int_0^T\|\psi(t)-\psi_*\|_{L^1(M)}\rmd t+\frac{1}{T}\int_0^T\|\nabla^\perp\psi(t)-\nabla^\perp\psi_*\|_{L^2(M)}\rmd t\Big|.\]

\begin{theorem}[Global Stability of Twisting]\label{twistingthm2}
Assume $\psi_*$ satisfies {\rm \textbf{Assumption 2}}. Let $u=\nabla^\perp\psi$ be any smooth velocity field on $M$ and $\Phi(t)$ its corresponding flow map, and set $\|\psi-\psi_*\|_{2}=\ve.$ Then, there exists a $C$ independent of $\ve$ and $t$ such that
\be\label{thetaclose}
\|\theta(\Phi(\cdot, t))- \theta_0 - \tfrac{2\pi}{\mu_*(\Phi(\cdot, t))} t\|_{L^2}  \leq C  \sqrt{\ve} t \qquad  \forall t\in \mathbb{R},
\ee
where $\theta$ is the lift to the universal cover of the annulus of the angular coordinate defined by \eqref{thetadef}.
In the case of a shear flow $u_*=(v_*,0)$ on the periodic channel $M=\mathbb{T}\times [0,1]$, the bound \eqref{thetaclose} becomes
\be\label{Phi1Close}
\|\widetilde{\Phi}_1(x,t)- x_1 - v_*(\widetilde{\Phi}_2(x,t)) t\|_{L^2}  \leq( \sqrt{2\pi\|v_*'\|_{L^\infty}} + \tfrac{1}{2} \sqrt{\ve})\sqrt{\ve} t \qquad  \forall t\in \mathbb{R},
\ee
where $\widetilde{\Phi}(\cdot,t)$ is the lift of the flow to the universal cover $\widetilde{M}= \mathbb{R}\times [0,1]$.
\end{theorem}

We note that both theorems require very weak control on the perturbation, which facilitates our application to nonlinear equations.

\begin{remark}
It is important to point out that the smallness assumption in Theorem \ref{twistingthm} on $\|\nabla^\perp\psi_*\cdot\nabla\psi\|_{L^1}$ is \textit{necessary} even in the autonomous case. Indeed, one can imagine (as in Figure \ref{crosshighway}) a large \textit{transversal} perturbation to the Hamiltonian that would introduce an impenetrable obstacle and prevent winding, even if it is tangentially localized. A large perturbation that is transversally localized would not have the same effect (since particles could simply swerve around any ``obstacle"). 
    	\begin{figure}[h!]
		\centering
		\includegraphics[width=0.225\linewidth]{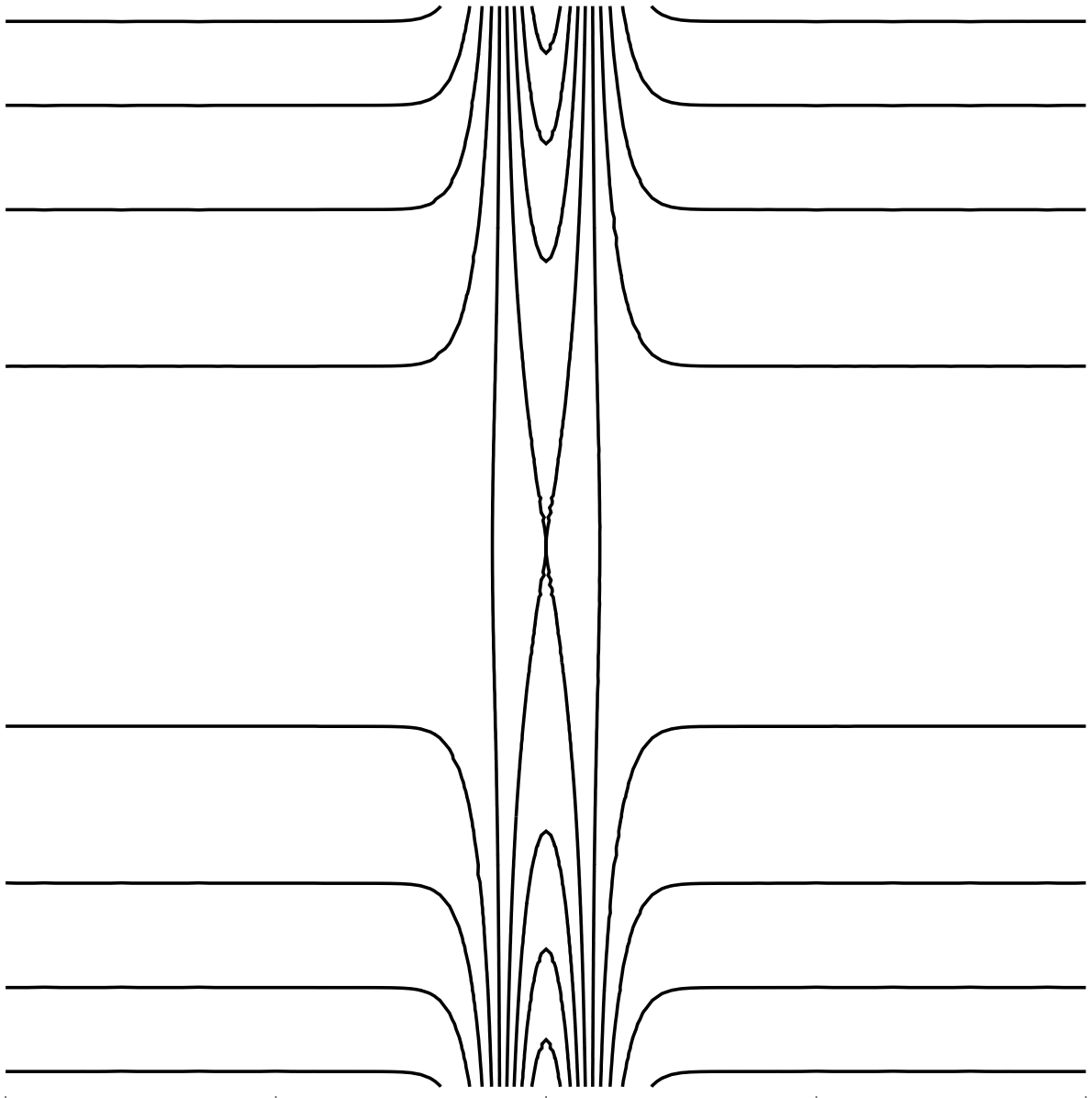}  \qquad  \qquad \qquad \includegraphics[width=0.32\linewidth]{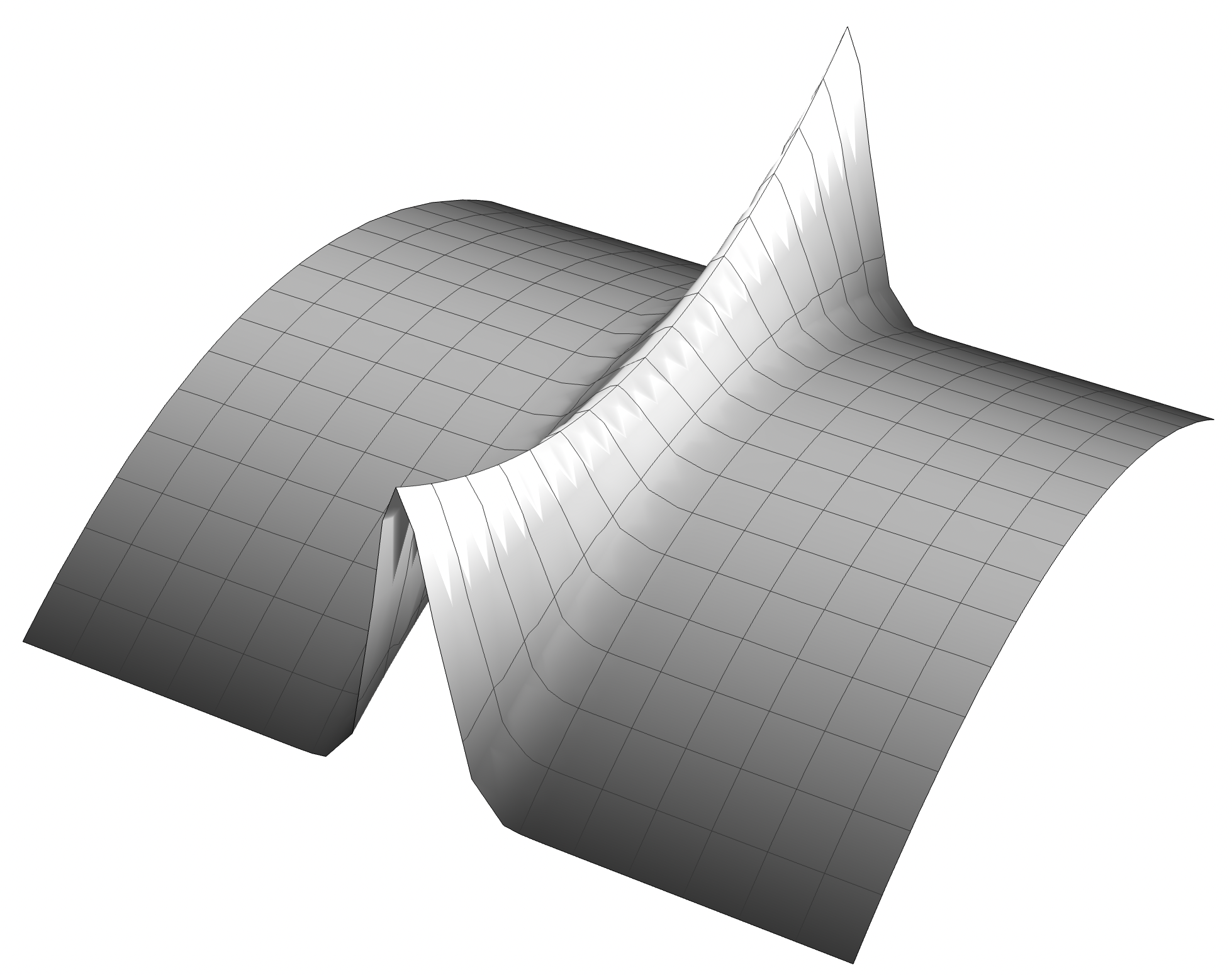}  
		\caption{Roadblocks in the highway by large perturbations; left is a contour plot, right is a landscape.} 
  \label{crosshighway}
	\end{figure}
	In addition, the bound \eqref{thetaclose} in Theorem \ref{twistingthm2} cannot hold in a pointwise sense, since single particles can invalidate the bound.  It is noteworthy that the $\sqrt{\ve}$ factor in \eqref{thetaclose} and \eqref{Phi1Close} is likely sharp for the $L^2$ estimate, while the corresponding $L^p$ estimate likely contains a factor of $\ve^{1/p}$ (with $\ve$ defined analogously). Getting the sharp bound in the case $p\in [1,2)$ appears to be more challenging than the case $[2,\infty)$. An interesting application related to the Arnold Diffusion is given in \S  \ref{arnolddiffsec}. 
\end{remark}

\subsection{Proof of the Stability of Twisting Theorem}

We prove Theorem \ref{twistingthm} for non-constant shear flow on $M=  \mathbb{T} \times [0,1]$.  This is because the computations are most transparent in this case, and the general setting follows a very similar argument. Suppose our shear profile is $u_*=(2\pi/\mu_*(x_2),0)$ where $\mu_*$ is the travel time defined by \eqref{traveltime}. Hereon we call $v_*(x_2):=2\pi/\mu_*(x_2)$ and name its stream function $\psi_*=\psi_*(x_2)$.  Since the shear is assumed non-constant, take any two values of $x_2$, say $y_1$ and $y_2$, such that $v_*(y_1)\neq v_*(y_2)$.  Without loss of generality, say that $v_*(y_1)>v_*(y_2)$. 

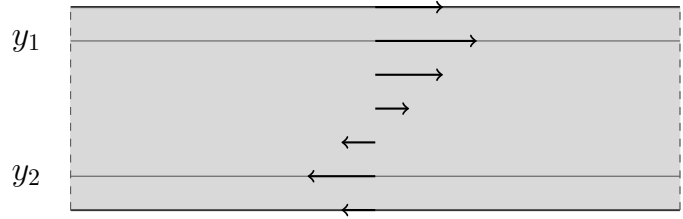
\begin{figure}[h!]
\centering
\begin{tikzpicture}[scale=0.9, every node/.style={transform shape}]
	\draw [thick] (-2.5,0)--(6.5,0);
	\draw [thick] (-2.5,3)--(6.5,3);	
	\draw[thin,
        postaction={decorate}] (-2.5,2.5)--(6.5,2.5);
	\draw[thin,  
        postaction={decorate}] (-2.5,0.5)--(6.5,0.5);
        \draw [dashed] (-2.5,0)--(-2.5,3);
        \draw [dashed] (6.5,0)--(6.5,3);
	\draw[fill, black!30!white, opacity=0.5] (-2.5,0)--(6.5,0)-- (6.5,3)--(-2.5,3)--cycle;
		               \draw[thick, black, ->] (2,.5)--(1,.5);
                  \draw[thick, black, ->] (2,1)--(1.5,1);
         \draw[thick, black, ->] (2,1.5)--(2.5,1.5);
         \draw[thick, black, ->] (2,2)--(3,2);
          \draw[thick, black, ->] (2,2.5)--(3.5,2.5);
          \draw[thick, black, ->] (2,3)--(3,3);
                         \draw[thick, black, ->] (2,0)--(1.5,0);                      
\draw  (-2.8,2.5) node[anchor=east] {\Large $y_1$};
\draw  (-2.8,0.5) node[anchor=east] {\Large $y_2$};
\end{tikzpicture}
  \caption{Cartoon of shear flow with two ``highways" localized near $y_1$ and $y_2$.}
  \label{fig:chan}
\end{figure}

We shall work on the universal cover of $M$ denoted by  $\widetilde{M}= \mathbb{R}\times [0,1]$.  We also consider the flowmap $\Phi(\cdot, t)$ (abusing notation by not writing $\widetilde{\Phi}(\cdot, t)$) as a periodic map on the universal cover 
\be\label{flowcond}
\Phi(x+2\pi e_1,t)= \Phi(x,t)+ 2\pi e_1.
\ee
We first establish the following key lemma that holds for general velocity fields on the channel:
\begin{lemma}\label{keylemma}
Let  $F:[0,1]\to \mathbb{R}$ be continuously differentiable and $u=\nabla^\perp \psi$.  Then we have:
\be
\int_{M} \Phi_1(x,t) F(\Phi_2(x,t))\rmd x     = \int_0^t \int_M u_1 (x_1,x_2,t)  F(x_2)\rmd x \rmd s + \mathsf{R}_F(t)
\ee
where the remainder $\mathsf{R}_F(t)$ obeys the bound
\be
|\mathsf{R}_F(t)| \leq 2\pi  \|F'\|_{L^\infty(0,1)} \int_0^t   \|u_2\|_{L^1(\mathbb{T}\times {\rm supp}(F))}\rmd s.
\ee
\end{lemma}

\begin{proof}
We compute the evolution of the quantity of interest:
\begin{align}\label{firstterm}
\frac{\rmd}{\rmd t} \int_{M} \Phi_1(x,t) F(\Phi_2(x,t))\rmd x &=  \int_{M} u_1(\Phi(x,t),t)F(\Phi_2(x,t))\rmd x\\
&\qquad  +  \int_{M} u_2(\Phi(x,t),t)\Phi_1(x,t) F'(\Phi_2(x,t))\rmd x . \label{secondterm}
\end{align}
In the first term, since the flowmap satisfies \eqref{flowcond}, we have by incompressibility
\be
 \int_{M} u_1(\Phi(x,t),t)F(\Phi_2(x,t))\rmd x =  \int_{\Phi(M,t)} u_1(x_1,x_2,t)F(x_2)\rmd x_1\rmd x_2=  \int_{M} u_1(x_1,x_2,t)F(x_2)\rmd x_1\rmd x_2
\ee
where $\Phi(M,t)\subset \widetilde{M}$ is interpreted as a subset of the universal cover.
The second equality (returning to the fundamental domain $M$) follows from the fact that $u_1$ is $2\pi$--periodic in $x_1$. 

For the second term \eqref{secondterm}, write $u_2= \partial_1 \psi$ where $\psi:M\to \mathbb{R}$ is a periodic function in $x_1$. Then
\begin{align}
\int_{M} u_2(\Phi(x,t),t)&\Phi_1(x,t)F'(\Phi_2(x,t))\rmd x = \int_{\Phi(M,t)} u_2(x,t)x_1F'(x_2)\rmd x_1\rmd x_2\\
&= \int_{\Phi(M,t)} \partial_1  ( \psi (x,t)- f(x_2,t)) x_1F'(x_2)\rmd x_1\rmd x_2\\
&= \int_{\Phi(M,t)} \partial_1 (x_1 \tilde{\psi} (x,t))F'(x_2)\rmd x_1\rmd x_2- \int_{\Phi(M,t)} \tilde{\psi} (x,t)F'(x_2)\rmd x_1\rmd x_2\\
&= \int_{\Phi(M,t)} \partial_1 (x_1 \tilde{\psi} (x,t))F'(x_2)\rmd x_1\rmd x_2- \int_{M}  \tilde{\psi} (x,t)F'(x_2)\rmd x_1\rmd x_2
\end{align}
where $\tilde{\psi}:= \psi(x_1,x_2,t) - f(x_2,t)$  for an arbitrary $f$.  We choose
$f(x_2,t) = \frac{1}{2\pi} \int_\mathbb{T} \psi (x_1, x_2,t) \rmd x_{1}$.  In the last equality, we finally used $\tilde{\psi}$ is periodic a periodic function of $x_1$. Thus we arrive at the identity
  \begin{align}
\frac{\rmd}{\rmd t} \int_{M} \Phi_1(x,t)   F(\Phi_2(x,t))\rmd x &= \int_0^1 \left( \int_{\mathbb{T}} u_1(x_1,x_2,t) \rmd x_{1} \right)  F(x_2)\rmd x_2 \\
&\qquad  +  \int_{\Phi(M,t)} \partial_1 (x_1 \tilde{\psi} (x,t))F'(x_2)\rmd x_1\rmd x_2.
\end{align}
To analyze this final term, we note that it is a total derivative in $x_1$.  As such, we write:
\be
\widetilde{M} \supseteq \Phi(M,t) = \bigcup_{c\in[0,1]} \Phi(M,t)  \cap\{x_2=c\}:=\bigcup_{c\in[0,1]}\Phi_c(M,t).
\ee
Since $\Phi: M \to \widetilde{M}$ is a diffeomorphism, generically for some odd $N(t;c)\in \mathbb{N}$ we have 
\be
\Phi_c(M,t) = \bigcup_{i=1}^{N(t;c)} [p_i^-(t;c), p_i^+(t;c)]
\ee
 where $p_i^\pm(t;c)\in \mathbb{R}$ are points on the line $\{x_2=c\}$ (this may fail for a measure zero set of $c\in[0,1]$).  Let us see why we may take $N$ to be odd.
The number of crossing of any curve connecting the top and bottom boundaries with a horizontal are, in general (if the tangent at the crossing is non-zero), an odd number. 
  If there is a tangency (non-generic) that induces no crossing, we regard this degenerate case as an empty interval that we ignore.
  Thus we have that $N(t;c)$ is odd. 

Note that, since $M$ is an annulus, $\Phi(M,t)$ has a left and right (free) boundary as a subset of the cover.
The points $p_i^-(t;c), p_i^+(t;c)$ can come from the image of either boundary, in principle.  However, what we know is that they come ordered in the following way: there are always an odd number of crossings with one boundary before crossings with another boundary occur (since crossing an even number of times with one boundary puts you outside the image of the fundamental domain).  Thus, if you start at $p_1^-$ (a left endpoint) - the first blue which is hit, call it $p_i$ must be its image from the other boundary.  Moreover, this image must be the right endpoint of an interval.  The fact that the orientation as an endpoint changes is crucial. 
See, e.g. Figure \ref{fig:imagewrapcover}.

\begin{figure}[h!]
\centering
\begin{tikzpicture}[scale=.8, every node/.style={transform shape}]
		   \draw [thick] (-8,0)--(12,0);
		   \draw [thick] (-8,3)--(12,3);
		     \draw [dashed] (12,0)--(12,3);
		   \draw [dashed] (-8,0)--(-8,3);
		   \draw [dashed] (-4,0)--(-4,3);
      		   \draw [dashed,red] (0,0)--(0,3);
		   \draw [dashed,blue] (4,0)--(4,3);
              \draw [dashed] (8,0)--(8,3);
	           \draw[fill, black!30!white, opacity=0.5] (-0,0)--(4,0)-- (4,3)--(-0,3)--cycle;

   \draw [very thick] (-4.25,1)--(-3.63,1);
   \draw [very thick] (-1.9,1)--(-0.22,1);
   \draw [very thick] (0.4,1)--(2.1,1);

  \draw [red, dashed, very thick] plot [smooth, tension=0.5] coordinates {(-3.5,0) (-4.5,1.5) (-2.5,0.5) (-1,2)  (1,2.5)  (1.5,3) };
  \draw [blue, dashed, very thick] plot [smooth, tension=0.5] coordinates {(-3.5+4,0) (-4.5+4,1.5) (-2.5+4,0.5) (-1+4,2)  (1+4,2.5)  (1.5+4,3) };

\draw  (-4.4,1.1) node[anchor=east] {\Large $p_1^-$};
\draw  (-2.6,1.1) node[anchor=east] {\Large $p_1^+$};
\draw  (-1.8,1.1) node[anchor=east] {\Large $p_2^-$};
\draw  (0,0.5) node[anchor=east] {\Large $p_2^+$};
\draw  (1,1.5) node[anchor=east] {\Large $p_3^-$};
\draw  (2.5,1.5) node[anchor=east] {\Large $p_3^+$};

  \draw  (-0.5,2.5) node[anchor=east] {\Large $\Phi(\{x_1=0\},t)$};
  \draw  (5,2.5) node[anchor=east] {\Large $\Phi(\{x_1=2\pi\},t)$};
  
\end{tikzpicture}
\hspace{-.022\textwidth}		
\caption{Covering space $\widetilde{M}=\mathbb{R}\times [0,1]$ of the cylinder ${M}=\mathbb{T}\times [0,1]$ . The fundamental domain $M$ in gray and the left/right boundary of its time $t$ image $\Phi(M,t)$ are red/blue dashed curves.}
\label{fig:imagewrapcover}
\end{figure}
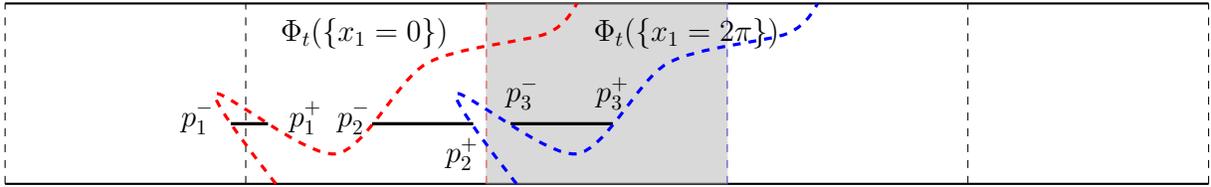
 \begin{align}
\int_{\Phi(M,t)} &\partial_1 (\tilde{\psi} (x,t)x_1)F'(x_2)\rmd x_1\rmd x_2= \int_0^1 \int_{\Phi_c(M,t)} \partial_1 (\tilde{\psi} (x,t)x_1)F'(c)\rmd x_1\rmd c\\
 &= \int_0^1 F'(c)\sum_{i=1}^{N(t;c)} \left[\tilde{\psi} (p_i^+(t;c),t) p_i^+(t;c)-\tilde{\psi} (p_i^-(t;c),t) p_i^-(t;c)\right]\rmd c\\
  &=   \int_0^1F'(c)\sum_{i=1}^{ N(t;c)}\left[ \tilde{\psi} (p_i^\pm(t;c),t)p_i^\pm(t;c)-  \tilde{\psi} (p_{i_*}^\mp(t;c),t) p_{i_*}^\mp(t;c)\right]\rmd c \\
  &=  2\pi \int_0^1F'(c)\sum_{i=1}^{ N(t;c)}(-1)^i\tilde{\psi} (p_i^\pm(t;c),t)\rmd c 
\end{align}
where $i_*$ is such that $p_i^\pm(t;c)-  p_{i_*}^\mp(t;c)=2\pi$ (the corresponding points on the boundaries.  In the above, 
we used that $\tilde{\psi} (p_i^+(t;c),t)=\tilde{\psi} (p_{i_*}^-(t;c),t)$ due to periodicity of $\psi$ on $M$.  Thus, we have
\be
\mathsf{R}_F(t):= 2\pi \int_0^t \int_0^1F'(c)\sum_{i=1}^{ N(s;c)}(-1)^i\tilde{\psi} (p_i^\pm(s;c),s)\rmd c \rmd s.
\ee
Whence, noting that $\int_\mathbb{T} \tilde{\psi}(x_1,x_2,t)\rmd x_1=0$, we have
\begin{align}
|\mathsf{R}_F(t)| &\leq 2\pi  \|F'\|_{L^\infty(0,1)} \int_0^t \left( \|\tilde{\psi}\|_{L^1(0,1;L^\infty(\mathbb{T}))} + \|\tilde{\psi}\|_{L^1(0,1;BV(\mathbb{T}))} \right)\rmd s\\ \label{Rfbound}
&\qquad  \leq  2\pi  \|F'\|_{L^\infty(0,1)} \int_0^t  \|u_2\|_{L^1(0,1;L^{1}(\mathbb{T}))}\rmd s.  
\end{align} This finishes the proof. 
\end{proof}

In fact, it can be seen from the above proof that we have the following:
\begin{corollary}\label{isotopycor}
    Let $\tilde{\varphi}$ be any lift (defined by an arbitrary isotopy to the identity) of an area-preserving diffeomorphism $\varphi$ of $M=\mathbb{T}\times [0,1]$ in the component of the identity, $g:M\to \mathbb{R}$ and $f:[0,1]\to \mathbb{R}$. Then
\be
    \left|\int_{M} (\partial_1 g)(\tilde{\varphi}(x))\tilde{\varphi}_1(x)f(\tilde{\varphi}_2(x))\rmd x \right| \leq |\mathbb{T}| \|f\|_{L^\infty(0,1)} \|\partial_1 g\|_{L^1(\mathbb{T})}.
\ee
In particular, the bound is \textit{independent} of the diffeomorphism $\varphi$.
\end{corollary}
Note that, naively one might expect that the bound should scale with ${\rm diam}(\tilde{\varphi}(M))$, the diameter image of the domain in the universal cover.  However, as our proof shows, there is a certain invariance at play resulting from strong topological constraints which is reflected in the isotopy independence of the upper bound.

\begin{proof}[Proof of Theorem \ref{twistingthm}]
Theorem \ref{twistingthm} will now follow from Lemma \eqref{keylemma} by choosing two functions $F,G:[0,1]\to \mathbb{R}$ with unit mass, localized to distinct ``highways". 
\be
\int_{M} (\Phi_1(x,t)-x)  F(\Phi_2(x,t))\rmd x    = \int_0^t \int_M u_1 (x_1,x_2,t)  F(x_2)\rmd x \rmd s + \mathsf{R}_F(t),
\ee
\be
\int_{M} (\Phi_1(x,t)-x) G(\Phi_2(x,t))\rmd x    = \int_0^t \int_M u_1 (x_1,x_2,t)  G(x_2)\rmd x \rmd s + \mathsf{R}_G(t),
\ee
whence
\begin{align*}
\int_{M} &\Phi_1(x,t) F(\Phi_2(x,t))\rmd x  -\int_{M} \Phi_1(x,t) G(\Phi_2(x,t))\rmd x   \\
&= \int_0^t \int_M u_1 (x_1,x_2,t)  F(x_2)\rmd x \rmd s - \int_0^t \int_M u_1 (x_1,x_2,t)  G(x_2)\rmd x \rmd s + \mathsf{R}_F(t)- \mathsf{R}_G(t).
\end{align*}
If  $\psi_*:= \psi_*(x_2)$ is a non-isochronal autonomous stream function of a shear flow on $M$ and
\be\label{psicondshear}
\frac{1}{T}\int_0^T \|\psi(t)- \psi_*\|_{L^{1}(M)} \rmd t  + \frac{1}{T}\int_0^T \|\partial_1 \psi\|_{L^{1}(M)} \rmd t\leq \ve  \qquad \forall T\in \mathbb{R},
\ee 
it follows that the remainders are bounded by
\be
|\mathsf{R}_F(t)- \mathsf{R}_G(t)|\leq 2\pi  \ve ( \|F'\|_{L^\infty(0,1)}+ \|G'\|_{L^\infty(0,1)} ) t,
\ee
while the main terms are
\begin{align}
\left|\int_0^t \int_M u_1 (x_1,x_2,t)  G(x_2)\rmd x \rmd s - \int_0^t \int_0^1 v_* (x_2,t)  G(x_2)\rmd x_2 \rmd s \right|&=\\
\left|\int_0^t \int_M \psi (x_1,x_2,t)  G'(x_2)\rmd x \rmd s - \int_0^t \int_0^1 \psi_* (x_2,t)  G'(x_2)\rmd x_2 \rmd s \right| &\leq \ve \|G'\|_{L^\infty(0,1)} t
\end{align}
where we used that $F$ and $G$ are compactly supported.  It they are furthermore localized to $y_1\neq y_2$, then we have
\be
\left|v(y_1)t- \int_0^t \int_0^1 \psi_* (x_2,t)  F'(x_2)\rmd x_2 \rmd s \right| \leq \ve t,\qquad \left|v(y_2)t- \int_0^t \int_0^1 \psi_* (x_2,t)  G'(x_2)\rmd x_2 \rmd s \right| \leq \ve t,
\ee
where $v_* = -\psi_*'$.
Thus, we find
\begin{align*}
\int_{M} \Phi_1(x,t) F(\Phi_2(x,t))\rmd x  -\int_{M} \Phi_1(x,t) G(\Phi_2(x,t))\rmd x  &= (v_*(y_1) -v_*(y_2)) t +  \mathsf{Rem}(t)
\end{align*}
where $|  \mathsf{Rem}(t)|\leq C\ve t$.  Choosing $\ve$ sufficiently small yields 
\be\label{growtht}
\int_{M} \Phi_1(x,t) F(\Phi_2(x,t))\rmd x  -\int_{M} \Phi_1(x,t) G(\Phi_2(x,t))\rmd x \sim (v_*(y_1) -v_*(y_2))  t
\ee
while, at the same time, each term grows individually:
\be\label{growtht2}
\int_{M} \Phi_1(x,t) F(\Phi_2(x,t))\rmd x  \sim v_*(y_1) t \qquad \int_{M} \Phi_1(x,t) G(\Phi_2(x,t))\rmd x  \sim v_*(y_2) t.
\ee

	\begin{figure}  
		\centering
		\includegraphics{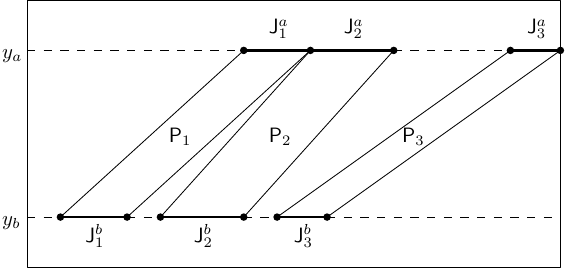}  
		\caption{Illustration for Lemma \ref{lem:flow-gradient-growth}. } \label{fig:gradient-L1}
	\end{figure}

This in turn implies  $L^1$ gradient growth:
\begin{lemma}\label{lem:flow-gradient-growth}
Let  $F,G:[0,1]\to \mathbb{R}$ be localized about $y_1\neq y_2$ respectively with non-overlapping supports.  Then if \eqref{growtht},\eqref{growtht2}  hold, then there is a constant $C>0$ such that $\|\nabla \Phi(\cdot,t)\|_{L^1} \geq C |t|$.
\end{lemma}
\begin{proof}
Note that the condition
\be
\left|\int_{M} \Phi_1(x,t) F(\Phi_2(x,t))\rmd x  -\int_{M} \Phi_1(x,t) G(\Phi_2(x,t))\rmd x  \right| \geq \frac{1}{2}|v_*(y_1) -v_*(y_2)|  t 
\ee
together with
\be
\left|\int_{M} \Phi_1(x,t) F(\Phi_2(x,t))\rmd x  \right| \geq \frac{|v_*(y_1)|}{2}t \qquad \left|\int_{M} \Phi_1(x,t) G(\Phi_2(x,t))\rmd x  \right| \geq \frac{|v_*(y_2)|}{2} t
\ee
implies that there exist sets $\mathsf{A}(t),\mathsf{B}(t)\subset M$ (defined by the set of points $x$ so that $\Phi_2(x,t) \in  {\rm supp}(F)$ and  ${\rm supp}(G)$ respectively) such that their measures satisfy $|\mathsf{A}(t)|, |\mathsf{B}(t)|>\mu$ for some $\mu>0$ independent of $t$ and the Euclidean distance on the cover satisfies
\be
{\rm dist}_{\Dm}(\Phi (x,t), \Phi (y,t)) \geq C t \qquad \text{for all} \quad x\in \mathsf{A}(t), \ y\in \mathsf{B}(t)
\ee
for an appropriate constant $C>0$. 

Note that $| \mathsf{A}(t) | > \mu$ implies \begin{equation*}
    \begin{split}
    \left| \left\{ y : | \mathsf{A}(t) \cap (\mathbb{T} \times \{ y \} )| > \mu/10 \right\} \right| > \mu/10. 
\end{split}
\end{equation*} A similar inequality holds for $\mathsf{B}(t) $ as well. Therefore, we may pick $y_{a}, y_{b} \in [0,1]$ with $|y_{a}-y_{b}|>\mu/20$ such that \begin{equation*}
     | \mathsf{A}(t) \cap (\mathbb{T} \times \{ y_{a} \} )| > \mu/10, \quad | \mathsf{B}(t) \cap (\mathbb{T} \times \{ y_{b} \} )| > \mu/10.
\end{equation*}
One can find a finite union of disjoint open intervals in $\mathbb{T} \times \{ y_{a} \}$, denoted by $\mathsf{I}^{a}_{1},\cdots, \mathsf{I}^{a}_{N}$ such that \begin{equation}\label{eq:thick1}
    |(\cup_{i=1}^{N} \mathsf{I}^{a}_{i}) \cap \mathsf{A}(t)| > \frac23 |\cup_{i=1}^{N} \mathsf{I}^{a}_{i}|.
\end{equation} Similarly, one may find disjoint open intervals in $\mathbb{T} \times \{ y_{b} \}$ with the property that  \begin{equation}\label{eq:thick2}
    |(\cup_{j=1}^{M} \mathsf{I}^{b}_{j}) \cap \mathsf{B}(t)| > \frac23 |\cup_{j=1}^{M} \mathsf{I}^{b}_{j}|.
\end{equation} 
For simplicity, let us further assume that $y_{a}>y_{b}$ and $|\cup_{i=1}^{N} \mathsf{I}^{a}_{i}| = \frac{\mu}{20} = |\cup_{j=1}^{M} \mathsf{I}^{b}_{j}|. $ By relabeling and subdividing these collections of intervals, one may ensure that for some $K$ and collections of disjoint intervals $\{ \mathsf{J}^{a}_{k}\}_{k=1}^{K}$, $\{ \mathsf{J}^{b}_{k}\}_{k=1}^{K}$ that \begin{equation*}
    \cup_{i=1}^{N} \mathsf{I}^{a}_{i} = \cup_{k=1}^{K} \mathsf{J}^{a}_{k}, \quad \cup_{j=1}^{M} \mathsf{I}^{b}_{j} =\cup_{k=1}^{K} \mathsf{J}^{b}_{k} 
\end{equation*} with the additional property that the intervals are ordered from left to right as $k$ increases, and that $|\mathsf{J}^{a}_{k}|=|\mathsf{J}^{b}_{k}|$ for each $k=1,\cdots, K.$

For each $k$, consider the parallelogram $\mathsf{P}_{k}$ whose top and bottom edges are given by $\mathsf{J}_{k}^{a}$ and $\mathsf{J}_{k}^{b}$, respectively (see Figure \ref{fig:gradient-L1}). For simplicity let us write $\mathsf{J}_{k}^{a} = [\alpha ,\alpha +\dlt] \times \{y_a\}$ and $\mathsf{J}_{k}^{b} = [\beta ,\beta +\dlt] \times \{y_b\}$ and parameterize $\mathsf{P}_{k}$ by $((1-\tau)\alpha + \tau\beta + \zeta, (1-\tau)y_{a} + \tau y_{b})$, with $0\le \tau \le 1$ and $0 \le \zeta \le \delta$. Furthermore, note that \begin{equation*} \begin{split}
    \Phi_{1}( \beta+\zeta , y_{b}) -  \Phi_{1}( \alpha+\zeta , y_{a}) & = \int_0^1 \frac{\rmd}{\rmd \tau} \left( \Phi_{1}((1-\tau)\alpha + \tau\beta + \zeta, (1-\tau)y_{a} + \tau y_{b}) \right)   \, \rmd \tau  \\
    & = \int_0^1 \left(   (\beta-\alpha) \partial_{1}\Phi_{1} + (y_{b}-y_{a}) \partial_{2}\Phi_{1} \right)  \, \rmd\tau 
\end{split}
\end{equation*} and now we can take absolute values and integrate both sides over a set $\mathsf{K}_{k}$ which is defined by the subset of $[0,\dlt]$ such that for $\zeta\in\mathsf{K}_{k}$, $(\beta+\zeta, y_b) \in \mathsf{B}(t)$ and $(\alpha +\zeta, y_a) \in \mathsf{A}(t)$. This condition guarantees that $|\Phi_{1}( \beta+\zeta , y_{b}) -  \Phi_{1}( \alpha+\zeta , y_{a}) |\gtrsim t.$ Then, we obtain that \begin{equation*}
    t| \mathsf{K}_{k} | \lesssim_{\mu} \nrm{ \nabla \Phi_{1} }_{L^1( \mathsf{P}_{k}) }
\end{equation*} where the implicit constant depends only on $\mu$. Summation in $k$ gives \begin{equation*}
    t\lesssim_{\mu} t\sum_{k=1}^{K}| \mathsf{K}_{k} | \lesssim_{\mu} \sum_{k=1}^{K}\nrm{ \nabla \Phi_{1} }_{L^1( \mathsf{P}_{k}) } \le \nrm{ \nabla \Phi_{1} }_{L^1(M) } 
\end{equation*} thanks to the conditions \eqref{eq:thick1}, \eqref{eq:thick2} (and using the pigeonhole principle) and the fact that the parallelograms $\mathsf{P}_{k}$ are non-overlapping. Lastly, note that $\nabla \Phi_{1}$ takes the same value when $\Phi_{1}$ is interpreted as a function onto $M$ (and not on the cover). This finishes the proof of Lemma \ref{lem:flow-gradient-growth}.
\end{proof}

In view of Lemmas \ref{keylemma} and \ref{lem:flow-gradient-growth}, the proof of Theorem \ref{twistingthm} for shear flows is now complete.
\end{proof}


We now prove Theorem \ref{twistingthm2}. 
\begin{proof}[Proof of Theorem \ref{twistingthm2}]
Compute
\begin{align}
\frac{\rmd}{\rmd t} \| \Phi_1(t) - x_1 - v_*(\Phi_2(t)) t \|_{L^2}^2  &= 2\int_M (u_1(\Phi(t)) - v_*(\Phi_2(t))) (\Phi_1(t) - x_1 -  t v_*(\Phi_2(t)))\rmd x\\
&\qquad -2 t \int_M \left[u_2(\Phi(t)) v_*'(\Phi_2(t)) (\Phi_1(t) - x_1 - t v_*(\Phi_2(t)) \right]\rmd x\\
& \leq 2 \| u_1 - v_* \|_{L^2} \| \Phi_1(t) - x_1 - v_*(\Phi_2(t)) t \|_{L^2}\\
&\qquad  -2 t \int_M \left[u_2(\Phi(t)) v_*'(\Phi_2(t)) (\Phi_1(t) - x_1 - t v_*(\Phi_2(t)) \right]\rmd x.
\end{align}
In the final term, we note
\be
\int_M u_2(\Phi(t)) v_*'(\Phi_2(t)) v_*(\Phi_2(t)) \rmd x = \int_M u_2(x_1,x_2) v_*'(x_2) v_*(x_2)\rmd x =0
\ee
 using that $u_2= \partial_1 \psi$.  We also bound 
\be
\left|\int_M u_2(\Phi(t)) v_*'(\Phi_2(t)) x_1 \rmd x\right|\leq  2\pi \|v_*'\|_{L^\infty} \|u_2(t) \|_{L^2}
\ee
which is immediate, as well as the same bound for
\begin{align}
\left|\int_M u_2(\Phi(t)) v_*'(\Phi_2(t)) \Phi_1(t)  \rmd x\right|&\leq  2\pi \|v_*'\|_{L^\infty}  \|{\psi}-\psi_*\|_{W^{1,1}}\\
&\leq 2\pi \|v_*'\|_{L^\infty}  \| u_2(t) \|_{L^2}
\end{align}
which follows from Corollary \ref{isotopycor} since we may write $u_2=\partial_1 (\psi -\psi_*)$.
Thus, with 
\be
q(t):= \| \Phi_1(t) - x_1 - v_*(\Phi_2(t)) t \|_{L^2}^2
\ee
we obtain the differential inequality
\be
\dot q(t) \leq  2\| u(t) - \nabla^\perp \psi_* \|_{L^2} \left(\sqrt{q(t)} +  4\pi   \|v_*'\|_{L^\infty}  t\right), \qquad q(0)=0.
\ee
We have
\be
\dot q(t) \leq  2\ve \left(\sqrt{q(t)} +  4\pi   \|v_*'\|_{L^\infty}  t\right), \qquad q(0)=0.
\ee
We claim $q(t)\leq \ve C^2  t^2$ for an appropriate constant $C$. 
First note that, by Taylor expansion,
\be
q(t)=  \| u_1(x_1,x_2,0)  - v_*(x_2)  \|_{L^2}^2 t^2 +O(t^3) \leq \ve^2 t^2 \qquad \text{as} \quad t\to 0.
\ee
Now, if $q(t)\leq \ve C^2  t^2$ then the right-hand-side is bounded as
\be
2\ve \left(\sqrt{q(t)} +  4\pi   \|v_*'\|_{L^\infty}  t\right)\leq  (C \ve^{3/2}  +   4\pi     \|v_*'\|_{L^\infty}  \ve) t  \leq 2 \ve C^2  t
\ee 
provided that $C$ is taken sufficiently large, e.g. $
C\geq \tfrac{1}{4} ( \sqrt{32 \pi \|v_*'\|_{L^\infty}+\ve} + \sqrt{\ve}).$  
\end{proof}

\subsection{A First Application: Quantitative Bounds on Arnold Diffusion}\label{arnolddiffsec}

The purpose of this section is to give a first application of Theorem \ref{twistingthm2}. Namely, we give quantitative bounds on the time it takes for a mass of particles to experience the so-called Arnold Diffusion. Our result is reminiscent of the work of Nekhoroshev \cite{Nek} for real-analytic stationary or time-periodic perturbations of integrable systems.  These results work also in higher dimension and establish stability of individual orbits over exponentially long periods by rather involved arguments \cite{Lochak}.  By contrast, our analysis applies to the movement of Lebesgue--typical trajectories rather than individual orbits, and the argument is a rather simple corollary of our stability of twisting Theorem \ref{twistingthm2}. The result applies to general nonautonomous $H^1$ perturbations where the techniques of KAM theory do not seem to apply. 

\begin{theorem}\label{arnolddiff}
Let $u_*=(v_*(y),0)$ be a shear flow on $M=\mathbb{T}\times [0,1]$.  Let $u=\nabla^\perp\psi$ be any (possibly non-autonomous) smooth velocity field on $\mathbb{T}\times [0,1]$ and let $\Phi(t)$ be its corresponding flow map. Set $\|\psi-\psi_*\|_{2}=\ve.$ Then for some $C_0>0$, we have
\be
\| v_*(\Phi_2(t)) - v_*(x_2)\|_{L^2} \leq C_0 \sqrt{\ve}\log(1+t).
\ee
\end{theorem}
\begin{remark}
If the shear is strictly monotone, then Theorem \ref{arnolddiff} implies that for some $C>0$
\be
\| \Phi_2(t) - x_2\|_{L^2} \leq C \sqrt{\ve}\log(1+t).
\ee
\end{remark}
\begin{proof}
Note that from the ODE $\dot{\Phi}_1(t) = u_1(\Phi(t),t)$, we have
\begin{align}
&{\Phi}_1(t)-(x_1 + t v_*(\Phi_2(t))) = \int_0^t u_1(\Phi(s),s)\rmd s -  t v_*(\Phi_2(t))  \\
 &\qquad = \int_0^t (v_*(\Phi_2(s))- v_*(x_2))\rmd s -  t (v_*(\Phi_2(t))- v_*(x_2)) + \int_0^t (u_1 - v_*)(\Phi(s),s)\rmd s.
\end{align}
Thus, the quantity $Q(t):=\|v_*(\Phi_2(t;\cdot))- v_*(\cdot)\|_{L^2}$ satisfies
\be
t Q(t) \leq \int_0^t Q(s)\rmd s + \mathcal{E}(t)
\ee
where
\be
\mathcal{E}(t):=\left\|{\Phi}_1(t)-(x_1 + t v_*(\Phi_2(t)) - \int_0^t (u_1 - v_*)(\Phi(s),s)\rmd s\right\|_{L^2}.
\ee
In view of Theorem \ref{twistingthm2}, this quantity satisfies $\mathcal{E}(t)\leq  C \min\{\sqrt{\ve}  t, \ve t^2\}$ (the quadratic behavior is important near $t=0$).  Thus we find
\be
\frac{\rmd }{\rmd t}\frac{1}{t}\int_0^t Q(s)\rmd s\leq \frac{\mathcal{E}(t)}{t^2}.
\ee
Integrating and using the bounds we see that 
\be
\frac{1}{t}\int_0^t Q(s)\rmd s\leq \int_0^t \frac{\mathcal{E}(s)}{s^2}\rmd s\leq C \sqrt{\ve}  \log(1+t).
\ee
The claim follows.
\end{proof}

\begin{remark}
Theorem \ref{arnolddiff} is sharp in the following sense: fix a Lipschitz shear profile $(v(y),0)$ on the channel $M=\mathbb{T}\times [0,1]$. Given a volume-preserving diffeomorphism $\Phi_*: M \to M$ in the component of the identity, there exists a vector field $u$ which is close to shear, e.g. $u=(v(y),0) + \ve b$ with $\|b\|_{L^\infty} \leq 1$, so that the corresponding flow map $\Phi(\cdot, t)$ begins at $\Phi(\cdot,0)= \Phi_*(\cdot)$ and is driven to $\Phi(x,y,T) = {\rm id} + T(v(y), 0)\  {\rm mod}\ 2\pi$ by a time $T\sim  e^{\|B\|_{L^\infty} (1+ \|v'\|_{L^\infty} )/ \ve} $, where $B$ is any divergence-free vector field defining an isotopy from identity to $\Phi_*$ in time 1.  This shows the ability to move an order 1 mass, initially distributed arbitrarily throughout the channel, to their appropriate vertical locations on the timescale of Theorem \eqref{arnolddiff}. The perturbation that accomplishes this movement is
    \be
b(x,y,t) =  \frac{\|B\|_{L^\infty}^{-1} (1+ \|v'\|_{L^\infty} )^{-1}}{(1+t)}  \begin{pmatrix} 
B_1( x- t v(y),y,s(t)) + t v'(y) B_2( x- t v(y),y,s(t)) \\
B_2( x- t v(y),y,s(t))
\end{pmatrix}.
\ee
\end{remark}

\subsection{Extensions}\label{sect:ext}

In the following, we give an indication of how to generalize the stability of twisting to Hamiltonian flows on different manifolds $M$.
All of the examples admit a global angular coordinate.  The general case of surfaces where a global angular coordinate is not available appears to require a new idea.

\subsubsection*{General annular surfaces}

It should be clear from the proof for shear flows that the computations are essentially local and therefore apply, as stated in  Theorem \ref{twistingthm}, to general annular domains $M$.
We may  use action-angle variables (see Arnold \cite{A2}), $(x,y)\mapsto (\psi_*,\theta)$ where $\psi_*$ is the ``radial coordinate" and  the ``angular coordinate" is $\theta:\Omega\to \mathbb{R}$ 
\be\label{thetadef}
\theta(x)= \frac{2\pi}{\mu(\psi_*(x))} \int_{\Gamma_{x_0(\psi_*),x} } \frac{\rmd \ell}{|\nabla \psi_*|} 
\ee
to deduce Theorem \ref{twistingthm}. We leave the details to the reader.

\subsubsection*{2D Tori}

For a two-dimensional torus, the universal covering space is the plane $\mathbb{R}^2$. In the proof of Theorem \ref{twistingthm}, we integrate over $[-\pi,\pi]\times \mathbb{R}$ rather than over the fundamental domain  $[-\pi,\pi]^2$.  With this modification, the image of the strip under the lifted flow has exactly two boundaries which must be fixed translates of one another as discussed in and around Figure \ref{fig:imagewrapcover}.

\subsubsection*{Punctured planar domains}\label{punctureddomains}
We can also consider domains such as  $M=\mathbb{R}^2/\{0\}$, namely we consider only those velocity fields on the plane that leave the origin invariant. We will give a bit more detail here since they can arise in planar Euler dynamics with some $m$-fold symmetries.  Indeed, we intend to use a version of Theorem \ref{twistingthm2} in \S \ref{perimgrow} to establish infinite perimeter growth for some vortex patches.

We have to be somewhat more careful in our computations since the domain is unbounded, but it will not be so different from the previous one. We shall consider the case where the velocities under consideration are near circular
\be
u_*(r,\theta)= v(r) e_\theta,
\ee
and let $\mu(r)=v(r)/r$. We will also assume sufficient decay on $u(x)$ when $|x|\rightarrow\infty.$ For simplicity, we will assume that all velocities under consideration decay on the order of $|x|^{-1}$ when $|x|$ is large. 
The more general setting is similar.
First observe that $\Phi_\theta$ and $\Phi_r$ satisfy:
\[\frac{\rmd}{\rmd t}\Phi_r=u_r(\Phi,t),\qquad \frac{\rmd}{\rmd t}\Phi_\theta = u_\theta(\Phi,t),\] where
\[u_\theta(r,\theta)=\frac{u\cdot x^\perp}{|x|^2},\qquad u_r(r,\theta)=\frac{u\cdot x}{|x|}.\]  Note that defining $u_\theta$ as above makes $\Phi_\theta$ and $\Phi_r$ analogous to $\Phi_1$ and $\Phi_2,$ respectively. 
Here, $\Phi(t)$ is understood as a diffeomorphism on the universal cover of $\mathbb{R}^2\setminus\{0\} \cong \mathbb{R}_+\times \mathbb{R}$.
Consider now
\[\lambda(t):=\int_{M} (\Phi_\theta-\theta) \mu(\Phi_r)\rmd x.\] Observe that this quantity is well-defined since $\Phi_\theta-\theta=\int_0^t u_{\theta} \rmd s$ decays for fixed $t$ like $\frac{1}{|\Phi|^2}$ for $\Phi$ large and similarly $\mu(\Phi_r)$ decays like $\frac{1}{|\Phi|^2}.$ This can be seen easily from Lemma \ref{VelocityDecay}. Now let us differentiate $\lambda.$ Observe that
\[\lambda'(t)=\int_{M} u_\theta \mu+\int_{M}(\Phi_\theta-\theta)\mu'(\Phi_r)u_r.\] Observe that both terms individually make perfect sense. Moreover, we may expand the second integral into another two integrals both of which are well defined:
\[\lambda'(t)=\int_{M} u_\theta \mu+\int_{M}\Phi_\theta \mu'(\Phi_r)u_r-\int_{M}\theta \mu'(\Phi_r)u_r.\]
The key is that, arguing exactly as in the Proof of Lemma \ref{keylemma}, the second term (which looks a-priori to be of order $\ve t$) can be estimated simply by $\|\tfrac{1}{r}u_r\|_{L^2},$ which should be assumed to be small, while we also assume that $r\mu'(r)\in L^2.$ 
Now considering:
\[\frac{\rmd }{\rmd t}\|\Phi_\theta-\theta-t \mu(\Phi_r)\|_{L^2}^2=2\int (\Phi_\theta-\theta-t \mu(\Phi_r))(u_\theta-\mu(\Phi_r)-t\mu'(\Phi_r)u_r).\]
It is now not difficult to show as before, under the assumptions mentioned above, namely that $r \mu',\mu,\frac{1}{r} u_r\in L^2$ that 
\begin{equation}\label{twistingthmunbounded}\|\Phi_\theta-\theta-t \mu(\Phi_r)\|_{L^2}\leq C(u_*)t\sqrt{\|\tfrac{1}{r}u_r\|_{L^2}+\|u_\theta-\mu\|_{L^2}}. \end{equation}

\section{Applications to two-dimensional perfect fluids}

\noindent 
We consider the Euler equations governing the motion of a fluid that is incompressible, inviscid and confined to a domain $M$,  possibly with boundary $\partial M$:
  \begin{alignat}{2}\label{eeb}
\partial_t u + u \cdot \nabla u &= -\nabla p,   &\qquad \text{in} \quad M,\\ \label{eeb2}
\nabla \cdot u &=0,   &\qquad \text{in} \quad M,\\ \label{eeb3}
u|_{t=0} &=u_0 ,   &\qquad\text{in} \quad M,\\
u\cdot \hat{n} &=0,   &\qquad\text{on}\   \partial M. \label{eef}
\end{alignat}
These equations are time reversible, meaning that $t\mapsto -t$ and $u\mapsto -u$ maps solutions to solutions. As such, any given ``structure" observed in the velocity field can appear at arbitrary late stages by preparing appropriate initial conditions.  Nevertheless, there appears to be one solid fact  borne out by numerical and physical experiment -- starting from ``any old" data, only a meager set of possible states, consisting generally of a few coherent vortices, persist indefinitely and represent some sort of weak attractor for the Euler system.
That is, the diversity of velocity fields appearing in the long time is greatly diminished as compared to their initial configuration, indicating a sort of gross entropy decrease or irreversibility in the Eulerian (describing the velocity field) phase space\footnote{For further discussions in this direction, see \cite{S3},  \S 9 of \cite{KMS22},  \S 3.4 of \cite{DE} and \cite{DD}.}. Understanding the mechanisms behind this {\it inviscid relaxation} and conjectured entropy decrease appears to be a major challenge, even though it is believed to occur for ``generic" solutions.  

The Lagrangian description, where we consider not the evolution of the velocity field but of the particles themselves, can give us some ideas into how such relaxation mechanisms can come about. 
In the Lagrangian picture, the configuration space of particle labellings is just the group of area-preserving diffeomorphisms $\Dm$ and Euler solutions represent parametrized paths $\{\Phi(\cdot, t)\}_{t\in \mathbb{R}}$ generated by the velocity vector field $u$ solving \eqref{eeb}--\eqref{eef}
\be\label{flowmap}
\frac{\rmd}{\rmd t}\Phi(\cdot, t)  = u(\Phi(\cdot, t) ,t) \qquad \Phi_0 = {\rm id}.
\ee
V.I. Arnold recognized \cite{A} that perfect fluid motion could be described as geodesic on $\Dm$, with respect to the $L^2$ metric (rigorously formalized by Ebin and Marsden \cite{EM}).  It is well known that the configuration space is vast in two-dimensions: diameter of the group $\Dm$ is infinite  \cite{ER}.  The dramatic decrease of diversity of the Eulerian velocity fields may be explained by the ample ``room" in the configuration space to absorb that lost complexity.\footnote{Indeed, this is one way to interpret the results on vorticity mixing and inviscid damping (e.g. coarse-grained relaxation to equilibrium) \cite{B}.}

 To further understand this point, it will be convenient to discuss the Eulerian state of the fluid by its  vorticity 
 $\omega := \nabla^\perp \cdot u$  with $\nabla^\perp:=(-\partial_2, \partial_1)$. The system \eqref{eeb}--\eqref{eef} can be reformulated as
  \begin{alignat}{2}\label{eevb}
\partial_t\omega + u \cdot \nabla \omega&=0 &\qquad \text{in} \quad M,\\
\omega|_{t=0} &=\omega_0 ,   &\qquad\text{in} \quad M, \label{eevbf}
\end{alignat}
where $u:=K_M[\omega]$ is recovered by the Biot--Savart law $K_M = \nabla^\perp \Delta^{-1}$ with $\Delta$ denoting the Dirichlet Laplacian on $M$ (if $M$ is simply connected.  For the more general case,  see  \S 2.2 of \cite{DE}.).
From \eqref{eevb}--\eqref{eevbf}, we see that the vorticity is, for all time, an area preserving rearrangement of its initial data
\be\label{vortformu}
\omega(t) = \omega_0\circ \Phi^{-1}(\cdot, t).
\ee
Fundamentally, it is this relation \eqref{vortformu} that is responsible for the transfer of information from the configuration space to the phase space, and vice versa.

In this work, we establish some qualitative expressions of irreversibility for the Euler equations that are essentially Lagrangian in nature and can be thought of as caused by, or symptoms of, the aforementioned vastness of the configuration space \cite{S3}. They will generally pertain only to solutions starting close to nearby Lyapunov stable equilibria $u_*$ which themselves induce some differential rotation (shearing) in their corresponding Lagrangian configuration (which is unsteady).

\subsection{Aging of two-dimensional perfect fluids}\label{sec:age}

Let $M\subset \mathbb{R}^2$ be a bounded domain with boundary $\partial M$.  Let $\Dm$ denote the group of area preserving diffeomorphisms on $M$. 
Here we quantify a certain increased complexity in the configuration space.  To do so, given a configuration $\varphi\in \Dm$, we can attribute an age:

\begin{definition}\label{agedef}
Let $\varphi\in \Dm$ and energy budget $\mathsf{E}>0$. The ``age" $t_{\mathsf{age}}(\varphi;\mathsf{E}) $ of $\varphi$ is
\be
t_{\mathsf{age}}(\varphi;\mathsf{E}) :=  \inf\left\{ T>0 \ :  \  \gamma_\cdot:[0,T]\mapsto  \Dm \quad \gamma_0=\id, \ \gamma_T=\varphi, \quad  \frac{1}{T}\int_0^T \|\dot{\gamma}_\tau\|_{L^2(M)}^2\rmd \tau \leq \mathsf{E}\right\}.
\ee
\end{definition}

An interesting question is, do fluids generically age? That is, as time runs on, does the configuration of a perfect fluid $\Phi(\cdot, t)$ become older in that $t_{\mathsf{age}}(\Phi(\cdot, t))\to \infty$ as $t\to\infty$? 
Here we prove that open sets of data give rise to aging Lagrangian configurations:

\begin{theorem}[Aging of the fluid]\label{agethm} Let $M$ be an annular surface (or any of the settings mentioned in \S\ref{sect:ext}).
Let $\omega_*$ be a non-isochronal $L^2$--stable steady state.  For $\ve:=\ve(\omega_*)$ sufficiently small, consider any smooth initial vorticity field $\omega_0$  having energy $\mathsf{E}_0=\frac{1}{2}\|u_0\|_{L^2}^2$ and satisfying $\|\omega-\omega_*\|_{L^2(M)}\leq \ve$.  Let  $\Phi(\cdot, t)$ be its corresponding  Lagrangian flow map at time $t$. Then $t_{\mathsf{age}}(\Phi(\cdot, t);\mathsf{E}_0)\geq C(\omega_*)t$ (in fact, ${\rm dist}_\Dm({\rm id}, \Phi(\cdot, t)) \geq C t$).  In particular, as $t\to\infty$ we have $t_{\mathsf{age}}(\Phi(\cdot, t))\to\infty$.
\end{theorem}

Mathematically, Theorem \ref{agethm} is the statement that, for any such $\omega_0$, the distance from the identity of the corresponding  flow map $\Phi(\cdot, t)$ measured by the $L^2$ metric (see discussion below) grows at least linearly in time.
The physical mechanism for this aging is indefinite twisting in the configuration space due to stable shearing, causing the  fluid to wrinkle in time.
 
Let us give some background. Fluid motion is governed by the ODE for $t \mapsto \Phi(\cdot, t)$ in $\Dm$:
\begin{equation}\label{ODiff} 
\begin{array}{ll}
\ddot \Phi(x,t) =-\nabla p\left(t,\Phi(x,t)\right) & \text{$(t,x)\in [0,T] \times M$,}\\
\Phi(x,0)=x &\text{$x\in M$,}\\
\Phi(\cdot,t)\in \Dm &\text{$t\in [0,T]$,}
\end{array}
\end{equation}
 The role of the acceleration (pressure gradient) is to confine the diffeomorphism $\Phi(\cdot, t)$ to $\Dm$.
One can view the configuration space $\Dm$  as an infinite-dimensional manifold with the
metric inherited from the embedding in $L^2(M;\mathbb{R}^2)$, and with tangent
space made by the divergence-free vector fields tangent to the boundary of $M$ \cite{EM}.  
Define the length of a path $\gamma_\cdot:[0,1]\mapsto  \Dm$  by
\be\label{length}
\mathscr{L}[\gamma]:= \int_{0}^{1} \|\dot{\gamma}_\tau(\cdot)\|_{L^2(M)} \rmd \tau.
\ee
See \cite{AK}.
We formally define the $L^2$ \textit{geodesic distance} connecting two states $ \varphi_0,\varphi_1\in \Dm$ by
\be
{\rm dist}_{\Dm}(\varphi_0,\varphi_1) = \inf_{\substack{ \gamma_\cdot:[0,1]\mapsto  \Dm\\
    \gamma(0)=\varphi_0, \ \gamma(1)=\varphi_1}}\mathscr{L}[\gamma].
 \ee
By the right-invariance of the $L^2$ metric, without loss of generality we take $\varphi_0= \id$ and call $\varphi_1=\varphi$.
 Arnold \cite{A} interpreted the Euler equation  \eqref{ODiff} for $t \mapsto \Phi(\cdot, t)$ as a \textit{geodesic} equation on $\Dm$, a critical point of the length functional $\mathscr{L}[\gamma]_{t_1}^{t_2}$. 
With this notion of distance, one can define the diameter of the group of area preserving diffeomorphisms as
\be
{\rm diam} (\Dm) := \sup_{\varphi_0,\varphi_1\in \Dm}{\rm dist}_{\Dm}(\varphi_0,\varphi_1).
\ee
One can think of ${\rm diam} (\Dm) $ as a measure of the capacity of $\Dm$ to store information. 
The celebrated result of Eliashberg and Ratiu \cite{ER} tells that
the diameter of $\Dm$ is infinite, resolving a conjecture of Shnirelman \cite{S} (who later gave an alternate proof \cite{S1}).
The notion of distance gives us a clock with which to measure the age of a diffeomorphism, introduced in Definition \ref{agedef}:

\begin{lemma} 
Fix $\varphi\in \Dm$  and let $t_{\mathsf{age}}(\varphi;\mathsf{E})$ be as in Definition \ref{agedef}.  Then
\be
t_{\mathsf{age}}(\varphi;\mathsf{E}) \geq {\rm dist}_{\Dm}(\varphi, \id)/ \sqrt{\mathsf{E}}.
\ee
\end{lemma}
\begin{proof}
By definition since, for any isotopy $\gamma_{\cdot}:[0,1]\mapsto  \Dm$ with $\gamma_0=\id$, $\gamma_1=\varphi$, we have
\be
{\rm dist}_{\Dm}(\varphi,\id )\leq \mathscr{L}[\{\gamma_\tau\}_{\tau\in[0,1]}] 
 \leq T \left(\frac{1}{T} \int_0^T \|\dot{\gamma}_{\tau/T}(\cdot)\|_{L^2(M)}^2  \rmd t \right)^{1/2} \leq T \sqrt{\mathsf{E}}
\ee
as $\gamma_{\cdot/T}:[0,T]\mapsto  \Dm$ with $\gamma_0=\id$ and $\gamma_T=\varphi$ is an isotopy with imposed energy $\mathsf{E}$.
\end{proof}

The result of Eliashberg and Ratiu \cite{ER} on the infinite diameter implies the existence of a sequence $\phi_n\in \Dm$ such that $t_{\mathsf{age}}(\varphi_n;\mathsf{E})\to \infty$ as $n\to\infty$ for any fixed energy $\mathsf{E}$.  The sequence $\phi_n$ is, in fact, constructed by iterating the time-1 flowmap (a twist map) of a stationary solution of the Euler equation.  
We prove here that aging is a stable feature of Euler.

\begin{proof}[Proof of Theorem \ref{agedef}]  We give an  elementary proof for domains that are annular, $M=\mathbb{T}\times [0,1]$.  The same result would follow from existing results in the literature  (e.g. Theorem 2.3 of \cite{S1} as well as the theorems of  \cite{BMS,M}) using the result of \S \ref{highway}.

By Lemma \ref{twistingthm}, any path of diffeomorphisms $\Phi_\cdot :\mathbb{R}\to \Dm$ generated by any velocity field nearby a non-isochronal steady field has the property that: 
\begin{itemize}
\item
there are two regions $A(t)$ and $B(t)$ on the universal covering space  $\widetilde{M}=\mathbb{R}\times [0,1]$ of  $M$ such that $A(0),B(0)\subset {M}$ and there is a $c_0>0$ such that for all time $t\in \mathbb{R}$:
\be
{\rm dist}_{\widetilde{M}}(A(t),B(t))\geq c_0 t.
\ee
where ${\rm dist}_{\widetilde{M}}$ is the Euclidean distance between the sets $A$ and $B$ in $\widetilde{M}$.
Moreover, each set contains a fixed positive measure set of points that started in the fundamental domain $M$: e.g. there are positive constants $c_A$ and $c_B$ so that for all $t\in \mathbb{R}$, 
\be
|\tilde{\Phi}(M,t)\cap A(t)| > c_A\qquad |\tilde{\Phi}(M,t)\cap B(t)| > c_B \label{boxes}
\ee
where $\tilde{\Phi}(\cdot, t)$ is the unique lift of the isotopy $\Phi_\cdot\cdot :\mathbb{R}\to \Dm$ to the universal cover.
\end{itemize}

\begin{figure}[h!]
\centering
\begin{tikzpicture}[scale=.8, every node/.style={transform shape}]
		   \draw [thick] (-8,0)--(12,0);
		   \draw [thick] (-8,3)--(12,3);
		     \draw [dashed] (12,0)--(12,3);
		   \draw [dashed] (-8,0)--(-8,3);
		   \draw [dashed] (-4,0)--(-4,3);
      		   \draw [dashed] (0,0)--(0,3);
		   \draw [dashed] (4,0)--(4,3);
        		   \draw [dashed] (8,0)--(8,3);
	           \draw[fill, black!30!white, opacity=0.5] (-0,0)--(4,0)-- (4,3)--(-0,3)--cycle;
	           
      		   \draw [blue, dashed] (10,2.5)--(11,2.5);
      		   \draw [blue, dashed] (10,2)--(11,2);
		    \draw [blue, dashed] (10,2)--(10,2.5);
		     \draw [blue, dashed] (11,2)--(11,2.5);
		     \draw[fill, blue, opacity=0.3] (10,2.5)--(11,2.5)-- (11,2)--(10,2)--cycle;
		     
		                             \draw[thick, black, ->] (11.25,2.28)--(11.75,2.28);
		     
		    \draw [red, dashed] (-7,1)--(-6,1);
      		   \draw [red, dashed] (-7,0.5)--(-6,0.5);
		    \draw [red, dashed] (-7,0.5)--(-7,1);
		     \draw [red, dashed] (-6,0.5)--(-6,1);
		      \draw[fill, red, opacity=0.3] (-7,1)--(-6,1)-- (-6,0.5)--(-7,0.5)--cycle;
		      
		                  \draw[thick, black, ->] (-7.27,0.78)--(-7.75,0.78);
		      
		               \draw[thick, black, ->] (2,.5)--(1,.5);
                  \draw[thick, black, ->] (2,1)--(1.5,1);
         \draw[thick, black, ->] (2,1.5)--(2.5,1.5);
         \draw[thick, black, ->] (2,2)--(3,2);
          \draw[thick, black, ->] (2,2.5)--(3.5,2.5);
          \draw[thick, black, ->] (2,3)--(3,3);
                         \draw[thick, black, ->] (2,0)--(1.5,0);

           \draw  (-6.15,1.5) node[anchor=east] {\Large $A$};
           \draw  (10.85,1.5) node[anchor=east] {\Large $B$};
\end{tikzpicture}
\hspace{-.022\textwidth}		
\caption{Covering space of the cylinder. The fundamental domain pictured in gray.}
\label{fig:imagewrapcover}
\end{figure}
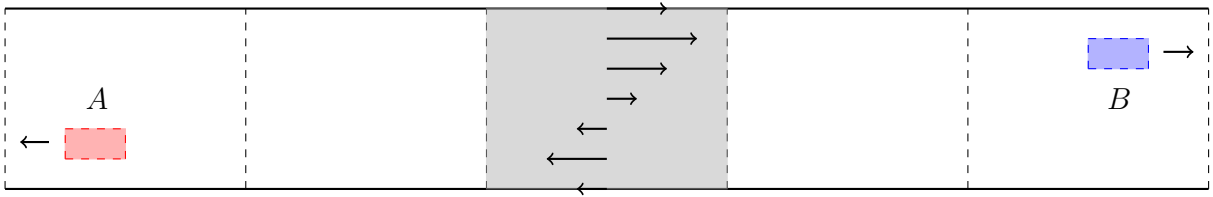

Fix now a time $T\in \mathbb{R}$ and a diffeomorphism $\Phi(\cdot, T)$. Consider any isotopy $\gamma_\cdot\cdot :[0,1]\to \Dm$ between $\id$ and $\Phi(\cdot, T)$.  Let $\tilde{\Phi}_T$ be \textit{any} lift of $\Phi(\cdot, T)$ to the universal cover.  
We estimate the length below by the $L^1$ length
\be
\mathscr{L}[\gamma] \geq \int_0^1 \int_M |\gamma_s(x)| \rmd x \rmd s =  \int_0^1 \int_M |\tilde{\gamma}_s(x)| \rmd x \rmd s
\ee
where $\tilde{\gamma}_t$ is unique  lift of the isotopy $\gamma_\cdot\cdot :[0,1]\to \Dm$ connecting $\tilde{\id}$ to $\tilde{\Phi}_T$.  The lower bound relates to the
$L^1$ length, which is just the average of the lengths of all paths $\{\gamma_s(x)\}_{s\in[0,1]}$.

Since any lift $\tilde{\Phi}_T$  can differ in the universal cover only  by a $2\pi$ multiple shift, in view of the property \eqref{boxes}, for either each $x\in \tilde{\Phi}_T(M)\cap A(t)$ or each $y\in \tilde{\Phi}_t(M)\cap B(t)$, the length of the curves is bounded below by the length of a straight line connecting the image point to the identity ${\rm len}(\tilde{\gamma}_\cdot(x))\geq c_0T$ so
\be\label{pathbound}
 \int_0^1 \int_M |\tilde{\gamma}_s(x)| \rmd x \rmd s\geq \min\{c_A,c_B\} c_0 T,
\ee
a bound which is independent of the connecting isotopy $\tilde{\gamma}$.  We complete the proof by noting
\be
{\rm dist}_{\Dm}({\Phi}_T,\id ) \geq \min\{c_A,c_B\} c_0 T.
\ee
\end{proof}

We end with the following\\

\noindent \textbf{Conjecture:}
Generic Euler flows $\Phi(\cdot, t)$ exhibit aging, e.g. $t_{\mathsf{age}}(\Phi(\cdot, t))\to\infty$.\\

It is not clear what notion of genericity should be taken. It might even hold in a very strong sense, such as: the only flows which do not exhibit aging are time periodic (isochronal) and these, in turn, are very rare.

\subsection{Complexity of the Lagrangian flowmap}

Next, we connect aging to a concept of complexity of the diffeomorphism introduced by Khesin, Misiolek and Shnirelman in
\cite{KMS22}.  Let $\exp_{\rm id}: T_{\rm id} \mathscr{D}_\mu^s (M)\to \mathscr{D}_\mu^s (M)$ be the $L^2$ exponential map, where $\mathscr{D}_\mu^s (M)$ is the Sobolev completion of the diffeomorphism group $\Dm$.  This map is a local diffeomorphism near the identity element ${\rm id}$ in $\Dm$. The Euler flow starting from initial velocity $u_0$ is $\Phi(t) = \exp_{\rm id}(t u_0)$.  Let us introduce an open neighborhood of radius $\ve$ in $H^s$:
\be
\mathsf{B}_\ve^s:= \{ v\in T_{\rm id} \mathscr{D}_\mu^s (M) \ : \  \|v\|_{H^s}< \ve\}.
\ee
Consider the set $\mathsf{U}_\ve= \exp_{\rm id}(\mathsf{B}_\ve^s)$ of all diffeomorphisms that can be reached by a time-one Euler flow with initial velocity in $\mathsf{B}_\ve^s$.  It can be proved that for any $\varphi\in \mathscr{D}_\mu^s (M)$ and each $\ve>0$, there exists a minimal finite $N_\ve:=N_\ve(\varphi)\in \mathbb{N}$ such that $\varphi$ can be represented exactly as a composition of finite number of elements in $\mathsf{U}_\ve$ (see \cite{Lukatskii1,Lukatskii2})
\be\label{decompostion}
\varphi = \eta_1 \circ \dots \circ \eta_{N_\ve}, \qquad \eta_i\in \mathsf{U}_\ve.
\ee
Define the \textit{complexity} of $\varphi\in \mathscr{D}_\mu^s (M)$ by 
\be
\mathcal{C}(\varphi) := \limsup_{\ve\to 0} (\ve N_\ve(\varphi)).
\ee
Khesin, Misiolek and Shnirelman conjectured (Problem 21 in \cite{KMS22}) that generic Euler flows will have $\mathcal{C}(\Phi(\cdot, t)) \simeq t$.  In what follows, we will establish a lower bound of this type.  To do so, we establish the following relation between complexity and the $L^2$ distance:
\begin{lemma}
Let $\varphi\in \mathscr{D}_\mu^s (M)$ be in the component of ${\rm id}$.  Then 
$\mathcal{C}(\varphi) \geq {\rm dist}_{\Dm}({\rm id},\varphi)$.
\end{lemma}
\begin{proof}
First note that by \eqref{decompostion}, we may define a time-one isotopy $\gamma_s:\mathscr{D}_\mu^s\to \mathscr{D}_\mu^s$ of $\varphi$ to the identity, e.g. $\gamma_0 = {\rm id}$ and $\gamma_1 =\varphi$, via
\be
\dot{\gamma}_s = v({\gamma}_s,s), \qquad v(x,s):= N_\ve \begin{cases}
    v_{N_\ve}(x,s) \qquad 0\leq t < \tfrac{1}{N_\ve}\\
   \qquad \vdots\\
    v_{1}(x,s) \qquad 1-\tfrac{1}{N_\ve} \leq t \leq 1
\end{cases}
\ee
where $v_i$ are the velocity fields in $\mathsf{B}_\ve^s$ which generate $\eta_i$.  As such $\|v_i\|_{L^2}\leq \|v_i\|_{H^s}<\ve N_\ve$ for each $i$. Then, by definition 
\begin{align}
{\rm dist}_{\Dm}({\rm id},\varphi) \leq \int_{0}^{1} \|\dot{\gamma}_\tau(\cdot)\|_{L^2(M)} \rmd \tau =\sum_{i=0}^{N_\ve-1}\int_{t_i}^{t_{i+1}} \|v_i(\cdot,\tau)\|_{L^2(M)} \rmd \tau  \leq \ve N_\ve,
 \end{align}
 since $t_{i+1}-t_{i}= \tfrac{1}{N_\ve}$. The claim follows.
\end{proof}
As a corollary of Theorem \ref{agethm}, we have
\begin{corollary}
    In the setting of Theorem  \ref{agethm}, the complexity of all Euler solutions grows
    \be
    \mathcal{C}(\Phi(\cdot, t)) \geq C t.
    \ee
\end{corollary}
It would be interesting also to connect this increased complexity to the possible lack of conjugate points along the flow, see e.g. \cite{Mconj,DM}.

\subsection{Filamentation, Spiraling and Wandering in Perfect Fluids}
This section is devoted to establishing some simple consequences of the Theorem \ref{twistingthm} related to the Eulerian picture of long-time behavior of solutions to the 2d Euler equation. 

\subsubsection{Generic loss of smoothness}

We begin by remarking that Theorem \ref{twistingthm} allows us to prove that the Eulerian vorticity of solutions nearby stable steady states generically becomes filamented:

\begin{theorem}[Generic Vorticity Filamentation]\label{gengrowth}
Let $M$ be an annular surface (or any of the settings mentioned in \S\ref{sect:ext}).
Fix $\alpha>0$ and let $\omega_*\in {C^{\alpha}(\overline{M})}$ be a non-isochronal stable Euler solution. Then there exists $\ve>0$ such that the set of initial data   
\be\label{Dset}
\left\{\omega_0\in B_\ve^{C^\alpha}(\omega_*)\ :  \  \sup_{t \ge 1 } \frac{\|\omega(t)\|_{C^\alpha}}{ |t|^{\alpha-}}= +\infty  \right\}
\ee is dense (in the strong $C^\alpha$ topology) in $ B_\ve^{C^\alpha}(\omega_*)$.
\end{theorem}

\begin{remark}
    If $M$ is the flat torus, there are no known Lyapunov stable steady states of the Euler equation. On the other hand, if one considers generic metrics on genus on tori, then the eigenspaces are simple \cite{U}, and thus there again exist infinite dimensional families of Arnold stable steady states to which our theorem may be applied. 
\end{remark}


This result validates a conjecture of Yudovich on generic deterioration of regularity for the Euler equations as it applies to neighborhoods of stable steady states \cite{Y1,Y2,MSY}.
The proof of Theorem \ref{gengrowth}  follows exactly the same lines as Corollary 3.27 of \cite{DE} (based on an idea of Koch \cite{Koch}, see also Margulis, Shnirelman and Yudovich \cite{MSY}) and therefore is not reproduced here.  It should be noted that, relative to Corollary 3.27 of \cite{DE} (which applies only to non-isochronal stable steady states on annular domains having different boundary velocities), there are no conditions in Theorem \ref{gengrowth} on the behavior on the boundary and the allowed domains $M$ are more general (in particular, we do not require they have boundary at all).  The key is that the stability of twisting established in our Theorem \ref{twistingthm} does not require differential boundary rotation. 

\subsubsection{Wandering}

The goal of this subsection is to use Theorem \ref{twistingthm}  to establish the existence of wandering neighborhoods in the $L^\infty$ topology for the 2d Euler equation nearby (in $L^2$) any Arnold stable steady.
 This was first proved by Nadirashvili nearby Couette flow \cite{N}. Our results are directly inspired by Nadirashvili's example and can be understood as a generalization of it to time-dependent annular regions constrained to be near two levels of
an appropriate stream function.   See also a generic wandering statement in an extended state space by Shnirelman \cite{S2} as well as the work of Khesin, Kuksin and Peralta-Salas \cite{KKPS} for 3d analogues.
 We restrict ourselves to the case of the periodic channel $\mathbb{T}\times [0,1]$ for simplicity of exposition. Our arguments can easily be extended to many other domains (see Section \ref{sect:ext}), even manifolds without boundary. 
In the following, we denote by $S_t:L^\infty \to L^\infty$ the solution map to the 2d Euler equation that sends a given initial data $\omega_0\in L^\infty(M)$ to its solution at time $t$. 

\begin{theorem}[Wandering Neighborhoods in 2d Euler]\label{wanderingthm}
Fix $M=\mathbb{T}\times [0,1]$. Let $\omega_*$ be an $L^2$ Lyapunov stable and non-constant shear flow on $M.$ Given any $\ve>0$, there exists a smooth $\omega$ with $\|\omega-\omega_*\|_{L^2}<\ve$ and an open neighborhood $\mathcal{U}$ in $L^\infty$ of $\omega$ with the property that
\[S_t(\mathcal{U})\cap \mathcal{U}=\emptyset,\] for any $t>C_*(\omega_*).$  That is, $\mathcal{U}$ is a wandering neighborhood. Moreover, every element of $S_t(\mathcal{U})$ exhibits at least linear growth of (i) the length of its level sets and (ii) its vorticity gradient.
\end{theorem}

\begin{proof}
    We fix $\omega_*$ and assume that $y_1,y_2\in (0,1)$ are such that $u_*(y_1)<u_*(y_2).$ Now take $\bar\omega_0$ to be within $\ve$ of $\omega_*$ in $L^2$ and having two level sets with distinct values, say $1$ and $2$, that are simple closed curves in $(-\pi,\pi)\times (0,1)$ and both enclosing the region $[-\pi+\kappa,\pi-\kappa]\times [\kappa,1-\kappa]$, for some $\kappa>0$ small\footnote{Strictly speaking, we could take the level sets to be straight lines connecting $y=0$ and $y=1$, but we avoid this since we are writing the proof to be applicable to situations without boundary.}.  Call these two level sets $\gamma_1$ and $\gamma_2$ and denote their ``insides" by $\Gamma_1$ and $\Gamma_2.$ Note that, for any small $L^\infty$ perturbation of $\omega,$ the values of $\omega$ on $\gamma_1$ are less $\frac{5}{4}$, while the values on $\gamma_2$ are strictly greater than $\frac{7}{4}$. Now fix $\omega_0\in B_{\ve}(\bar\omega_0)$ (the $L^\infty$ ball). Let $\Gamma_i(t)$ be the image under the (lifted) flow $\Phi(\cdot, t)$ associated to $S_t(\omega_0)$.  Taking $\ve$ sufficiently small and $t$ sufficiently large, we deduce by  the Lyapunov stability of $\omega_*$, the continuity of $u_*$, and Theorem \ref{twistingthm} , the existence of two positive measure sets $A(t),B(t)\subset M$, with $|A(t)|,|B(t)|>c_1$ and 
    \[\Phi_1(z,t)<t\left(u_*(y_1)+\frac{u_*(y_2)-u_{*}(y_1)}{4}\right)\qquad y_1-c_2<\Phi_2(z,t)<y_1+c_2\qquad \forall z\in A(t),\]
     \[\Phi_1(z,t)>t\left(u_*(y_2)-\frac{u_*(y_2)-u_{*}(y_1)}{4}\right)\qquad y_2-c_2<\Phi_2(z,t)<y_2+c_2\qquad \forall z\in B(t),\] with $c_1$ and $c_2$ small independent of $\kappa$ and $\ve$, provided $\ve$ is sufficiently small. Now observe that if $\kappa$ is sufficiently small, we have
     \[A(t)\cap \Gamma_i\not=\emptyset\not=B(t)\cap \Gamma_i \qquad  \text{for} \ i=1,2.
     \]
     Whence, it is not difficult to conclude that the length of the level curves $\Phi(\gamma_i,t)$ grows at least linearly in time and that
     \be\label{Closeness}\inf_{z_1\in\gamma_1, z_2\in\gamma_2}|\Phi(z_1,t)-\Phi(z_2,t)|\leq \frac{C}{t},\ee for all $t$ sufficiently large (see for example the proof of Lemma 3.15 of \cite{DE}). This, in particular implies that $\|\nabla S_t(\omega_0)\|_{L^\infty}\geq ct$ for all $t$ sufficiently large for every $\omega_0\in B_\ve(\bar\omega_0).$ Moreover, by \eqref{Closeness} it follows that \[\|\bar\omega_0-S_t(\omega_0)\|_{L^\infty}>\frac{1}{4}\] for all $t$ sufficiently large, since $\bar\omega_0$ is continuous.  This concludes the proof of wandering. 
\end{proof}

\subsubsection{Unbounded gradient growth in the SQG equation}
For another application of Theorem \ref{twistingthm}, consider the family of equations:
\be\label{SQG1}
\partial_t \omega + u\cdot\nabla\omega=0,
\ee
\be
\label{SQG2}
u=\nabla^\perp (- \Delta)^{- \alpha} \omega 
\ee on $\mathbb{T}^2$. Here $\alpha \in [\frac{1}{2},\infty)$ is a parameter that interpolates between the 2d Euler equation $(\alpha=1)$ and the SQG equation $(\alpha=\frac{1}{2}).$
It is well known that the SQG equation is locally well-posed in the Sobolev spaces $H^s$ when $s>2$. An open problem is to establish unbounded gradient growth in the SQG equation in finite or infinite time (see the statements and discussion in \cite{KiselevActiveScalars,HeKiselev,KN,DenisovSuperLinear,CJ2,Z,KS}). Arbitrarily large, but finite, growth was established by Kiselev and Nazarov \cite{KN} while infinite growth for higher norms was established recently by He and Kiselev \cite{HeKiselev}. A major difficulty in establishing growth is the relatively weak control that the conservation laws of $\omega$ give on the velocity field $u.$ In particular, we do not even have an \textit{a-priori} pointwise bound on the velocity field $u$. An important aspect of the stability of twisting theorem \ref{twistingthm2} is that we only require $L^2$ control on the velocity field. Relying on this, we will be able to establish unbounded gradient growth in the SQG equation with relative ease.

Let us start with some basic observations. The first is that the steady state $\sin(y)$ is stable in the sense of Arnold \cite{A,AK}, once we restrict to a symmetry class.  

\begin{lemma}
In the class of classical solutions $\omega$ that are odd in $y$ on $\mathbb{T}^2,$ the steady state $\omega_*=\sin(y)$ is Lyapunov stable in $L^2.$
\end{lemma}
The proof is standard and is based simply on the conservation of $\Vert\omega\Vert_{L^2}$ and $\Vert(-\Delta)^{-\frac{\alpha}{2}}\omega\Vert_{L^2}$ (see for example, \cite[Lemma 2.1]{DenisovSuperLinear}). We can now state and prove the main theorem of this subsection.

\begin{theorem}\label{sqgthm}
Fix $\alpha\geq \frac{1}{2}$. Assume that $\omega_0$ is odd in $y$ on $\mathbb{T}^2$, sufficiently close in $L^2$ to $\sin(y)$, and has two distinct level sets in $[-\pi,\pi]\times [0,\pi]$ with different values and enclosing areas sufficiently close to $2\pi^2$. Then, the unique solution $\omega$ to \eqref{SQG1} either becomes singular in finite time or there exists a constant $c>0$ so that
\[\|\nabla\omega(t)\|_{L^1}\geq c t \qquad \text{as} \qquad t\to\infty.\] 
\end{theorem}
\begin{remark}
    Let us remark that the $L^1$ growth of the gradient appears to be new even for the Euler equation on $\mathbb{T}^2$ (the case $\alpha=1$).
\end{remark}

\begin{proof}
First note that if $\Vert\omega-\sin(y)\Vert_{L^2}\leq \ve $, then $\Vert u-(\cos(y),0) \Vert_{L^2}\leq  \ve .$ This estimate follows directly from the definition of $u$, given $\omega$ given in \eqref{SQG2} and is the sole reason for the restriction $\alpha\geq \frac{1}{2}$.  The proof proceeds as in that of Theorem \ref{wanderingthm}, but relying on Theorem \ref{twistingthm2} to produce the requisite fixed area regions $\mathsf{A}(t)$ and $\mathsf{B}(t)$. In particular, we make a small $L^2$ perturbation of $\sin(y)$ keeping the odd-in-$y$ symmetry on $\mathbb{T}^2$ and ensuring that the resulting $\omega_0$ has two different level sets in $[-\pi,\pi]\times [0,\pi]$ enclosing areas close to $2\pi^2.$

Let us denote these level sets by $\gmm_1$ and $\gmm_2$, with $\gmm_2$ enclosing a larger region. Assume for simplicity that $\omega_0 = 1$ on $\gmm_1$ and $=2$ on $\gmm_2$. Using Theorem \ref{twistingthm2}, it is easy to then deduce that there are positive mass regions of particles coming from inside the area enclosed by the level sets that become very far when the dynamics is lifted to $\mathbb{R}\times [0,\pi].$ It then follows that the diameter of the image of $\gmm_i$ in the cover by the flow map is at least $ct$ for some universal constant $c>0.$

We now note that \begin{equation}
    \begin{split}
        \Vert \nabla\omega(t)\Vert_{L^1} \ge \int_{-\pi}^{\pi} \left[ \int_0^{\pi} |\partial_y \omega(x,y,t)| \, \rmd y \right]  \rmd x 
    \end{split}
\end{equation} and the integral $$\int_0^{\pi} |\partial_y \omega(x,y,t)| \, \rmd y $$ is simply bounded from below by the number of alternating intersections that the line $\{x\} \times (0,\pi)$ makes with $\gamma_1$ and $\gamma_2$.   The result follows from the topological lemma stated below. 
\end{proof}

\begin{lemma}\label{lem:top}
    Let $\gmm:[0,1]\to \mathbb{T}\times (0,\pi)$ and $\eta: [0,1]\to \mathbb{T}\times (0,\pi)$ be non self-intersecting smooth curves, except for $\eta(0)=\eta(1)$. Assume that the lift of the curves $\tilde{\gmm}$ and $\tilde{\eta}$ into the cover $\mathbb{R}\times[0,\pi]$ has the following properties: 
    \begin{itemize}
        \item $\mathrm{diam}_{x}(\tilde{\gmm}) \ge 2\pi M$ for some integer $M\ge1$; namely, there are two points on $\tilde{\gmm}$ whose distance in the $x$-direction is at least $2\pi M$.
        \item The region of $\mathbb{R}\times[0,\pi]$ enclosed by $\tilde{\eta}$ contains $\tilde{\gmm}.$
    \end{itemize}
    Then, we have that the image of  $\gmm$ viewed on the square $[-\pi,\pi]\times (0,\pi)$ cuts it into connected components, at least $M+1$ of which is traversed by the image of $\eta$.  
\end{lemma}

The proof is not difficult and simply uses the intermediate value theorem on the cover. See Figure \ref{fig:torus} for an illustration in the case $M=3$. Blue and red curves correspond to images of $\eta$ and $\gmm$, respectively. 
	
	\begin{figure} [h!]
		\centering
		\includegraphics[height=0.2\linewidth]{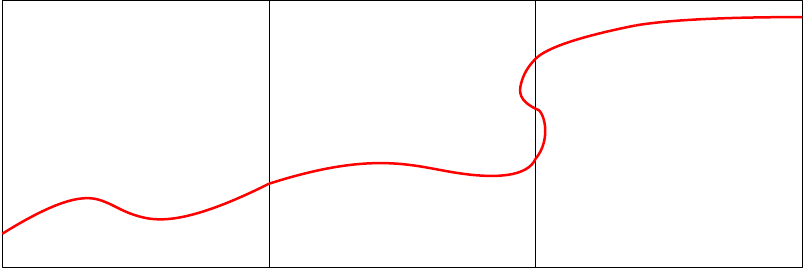}  \quad \includegraphics[height=0.2\linewidth]{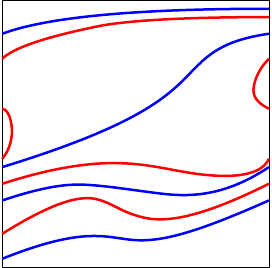}  
		\caption{Illustration for Lemma \ref{lem:top}. Each ``essential'' connected component of the domain defined by $\gmm$ (red curve) is traversed by $\eta$ (blue curve).} \label{fig:torus}
	\end{figure}

\subsection{Spiraling vortex patches with unbounded perimeter growth}\label{perimgrow}

In this section, we present examples of vortex patches in $\mathbb{R}^{2}$ whose perimeter grows at least linearly for all times. A vortex patch is a weak solution to the 2d Euler equation \eqref{eevb} where the vorticity $\omega$ is the characteristic function of a bounded set. A vortex patch is called smooth if the boundary of the set is smooth. A classical result established in \cite{Chemin} and later in \cite{BC} and \cite{Serfati} is that initially smooth vortex patches remain smooth for all time. The upper bound given in \cite{BC} for the patch perimeter is double exponential in time, and it has not been improved so far.  The long-time behavior of vortex patches remains essentially wide open, though there has been some progress in recent years. 

A phenomenon that we are interested in is the filamentation of vortex patches and their axi-symmetrization. While many numerical simulations exist indicating that vortex patches tend to filament and develop unbounded perimeter (see, for example, \cite{MMZ} and \cite{D} and numerous later works), we are not aware of any rigorous proofs of these phenomena. Indeed, a major conceptual barrier to establishing such behavior rigorously is the existence of a large class of steady and purely rotating states \cite{HMV,CCG} in addition to more recently constructed quasi-periodic states \cite{BHM,HHM,GIP}. We must emphasize at this point that we are considering  \textit{single signed} vortex patches. Simple examples of filamentation where positive and negative regions of the vorticity separate, mimicking oppositely signed point vortices, and leave behind long tails are not possible when the vorticity is signed on $\mathbb{R}^2$ since vorticity of a single sign will tend to wind around rather than spread out (see \cite{Marchioro}). 

The two issues we have just described, the plethora of coherent structures as well as the confinement of vorticity, make it difficult to conceptualize a proof of filamentation in vortex patches. How does one ensure that a given initial patch does not slowly migrate toward one such state while perimeter growth and filamentation occurs for long-time but not all time?  We resolve this problem with two ideas. First, guided by the stability of global twisting given to us from Theorem \ref{twistingthm2}, we search for a patch that starts close to the so-called Rankine vortex, which is simply the characteristic function of the unit disk, $\chi_{B_1}$, whose associated velocity field is twisting. The Lyapunov stability of the Rankine vortex will ensure that twisting occurs for some particles in $\mathbb{R}^2.$ Since we seek to show that the boundary of the patch twists (and not just the image of the fundamental domain) and since the boundary of the patch has Lebesgue measure zero, it does not appear at first glance that we can apply Theorem \ref{twistingthm2} to say anything about the boundary of a patch, no matter how close it is to $\chi_{B_1}.$ In order to ensure that the boundary of the patch itself must become wound up, we introduce a large hole into the patch that allows the whole patch to be close to $\chi_{B_1}$ in $L^1$ but that ensures that there must be a mass of particles holding back the boundary of the patch. In this way, by continuity, and aided by Theorem \ref{twistingthm2} we can ensure that there are points of the boundary of the patch with fixed distance from the origin and with completely different winding. From here, there is a somewhat involved but elementary argument ensuring that the perimeter of the patch must become large. It would be interesting to see if filamentation and perimeter growth can be established for simply connected patches (i.e. without using a large hole); this would probably require several new ideas. 

\subsubsection{Specification of the patch}

    We now make specifications on the type of initial patch we will consider. 
	\medskip
	
	\noindent \textbf{Assumptions on the initial patch}. We take $\Omega_{0}$ to be an open set which is 4-fold symmetric (invariant under the $90\deg$ rotation around the origin) with four connected components. 
	Take some $N\gg1$ and $0<\dlt \ll 1$. The set $\Omega_{0}$ will consist of three parts: \begin{itemize}
		\item \textit{Bulk}: the unit disc $B_{1}$ minus the $\dlt$-neighborhood of the diagonals $\{ x_1 = \pm x_2 \}$   
		\item \textit{Bridges}: the $\dlt$-neighborhood of the $x_1$/$x_2$ axes intersected with the annulus $\{ 1 < |x| < N \}$.
		\item \textit{Rings}: Four annular regions with area $\delta$ and holes of area $N^2.$ 
	\end{itemize}
	See Figure \ref{fig:patch-perimeter} for a schematic (but not to scale) description of $\Omega_{0}$. The key is that the measure of $\Omega_0\Delta B_1$ is small and $\Omega_0$ consists of four connected pieces with large holes. Furthermore, $\Omega_0$ contains no circle centered at $0.$ Observe that we may perturb the patch $\Omega_{0}$ described above and in the figure in $L^1$ as long as it is still 4-fold symmetric and homeomorphic to $\Omega_0$. In particular, the patch boundary can be smooth or even analytic. Then we have the following result. 
	
	\begin{theorem}[Infinite Perimeter Growth for Patches]\label{thm:patch}
		Let $\Omega(t)$ be the patch solution associated with $\Omega_{0}$. There exists a universal constant $c_0>0$ so that $|\partial\Omega(t)|\geq c_0 t$ for all $t\geq 0.$ 
 In fact, any lifting of $\Omega(t)$ satisfies that its diameter on $\mathbb{R}\times\mathbb{R}^+$ is bounded from below by $c_0 t.$
	\end{theorem}

 \begin{remark} As we discussed above, the main tool in establishing the theorem is the version of the global stability of twisting Theorem \ref{twistingthm2} discussed in Section \ref{punctureddomains}, though it is possible that the full might of Theorem \ref{twistingthm2} is not needed to establish the statement of the above theorem for this type of patch (Theorem \ref{twistingthm2} gives a bit more information than what is stated above). We remark that there were previous rigorous results in this direction giving long, but finite, time perimeter growth \cite{CJ} and infinite-time spiraling \cite{EJSVPII}. Though the result of \cite{EJSVPII} is an infinite time one, it only applies to a non-smooth patch and the local nature of the argument does not necessarily give any perimeter growth. It is also important to point out that the argument given here actually implies that the different components of $\Omega(t)$ ``entangle'' as $t\rightarrow\infty$ in the sense that any straight line passing through the origin will intersect the boundaries of the components of the patch $O(t)$ times as $t\rightarrow\infty$; moreover, the number of consecutive crossings with the boundaries of the different components is also of $O(t)$ as $t\rightarrow\infty.$ This is simply a consequence of the growth of the diameter of the lift of the patch. This entangling was observed numerically \cite{MMZ} in addition to the filamentation and perimeter growth.
 \end{remark}
	
	\begin{figure}[h!]
		\centering
		\includegraphics[width=0.36\linewidth]{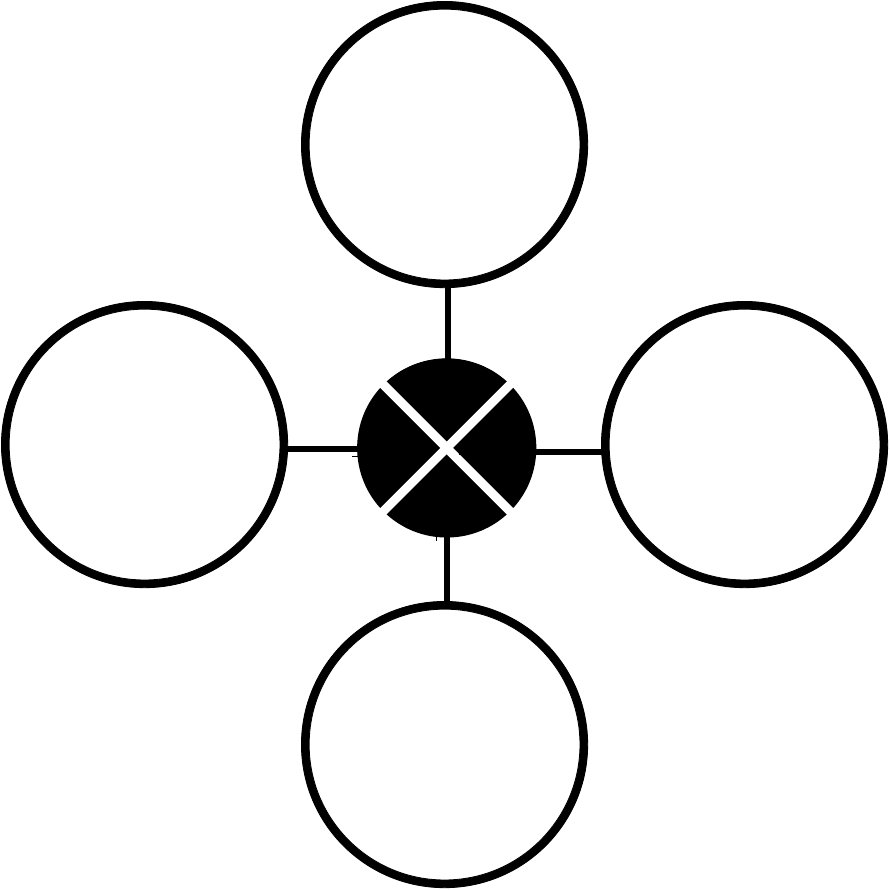}  \qquad \includegraphics[width=0.36\linewidth]{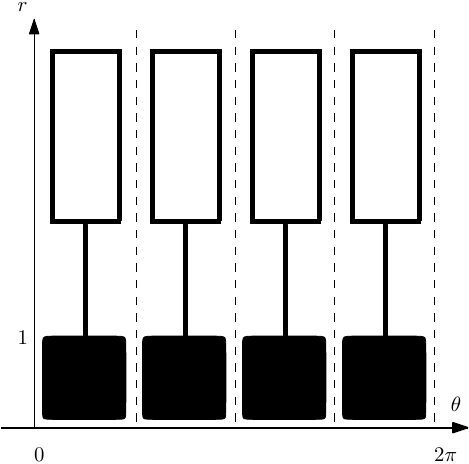}  
		\caption{Left: Black region  corresponds to the interior of $\Omega_{0}$.  Right: A (rough) image of $\Omega_{0}$ in polar coordinates.} \label{fig:patch-perimeter}
	\end{figure}

	\begin{remark}
	    Recently in \cite{CJL}, linear growth of perimeter was obtained for patch solutions defined in the cylinder $\mathbb{T} \times [0,\infty)$ which are \textit{attached} to the boundary $\mathbb{T}\times \{ 0 \}.$ In this work, the assumption that the patch is attached to the boundary was essential; it was used to guarantee that there exist some ``slow'' particles on the patch, while it is not difficult to see that on the average all the particles are moving at a faster speed. Employing our main theorem, this assumption can be removed, and one can also obtain infinite perimeter growth for patches defined in $\mathbb{T} \times \mathbb{R}.$
	\end{remark}
	
	
    The remainder of this section is devoted to the proof of Theorem \ref{thm:patch}. In \S \ref{subsec:winding-to-perim}, we reduce the statement to proving existence of a pair of patch boundary points with a large gap in the respective winding numbers around the origin, by establishing a geometric lemma which could be of independent interest (Lemma \ref{lem:geom}). Then, the proof will be completed in \S \ref{subsec:finish} based on \S \ref{sect:ext} after recalling the stability statement of the Rankine vortex.  
 
	\subsubsection{From winding to perimeter growth}\label{subsec:winding-to-perim}
	
	We reduce the problem of perimeter growth to the one showing the existence of a pair of points, located on the patch boundary, with large and small winding numbers. Let us start with a basic definition:
	\begin{definition}
		In this section, as in Section \ref{punctureddomains}, $\Phi_\tht, \Phi_r$ are defined on the cover $(\tht,r) \in \mathbb{R}\times \mathbb{R}_+$. The corresponding velocities are $u_\tht$ and $u_r$. 
	\end{definition}

    Now we state the key proposition that we will establish using \ref{twistingthm2}:

	\begin{proposition}\label{prop:patch-key}
		For any $t\ge0$, there exist points $\bar{x}_{1}(t)$ and $\bar{x}_{2}(t)$ on $\partial\Omega_{0}$ which satisfy \begin{equation}\label{eq:x1}
			\begin{split}
				\Phi_{\theta}(\bar{x}_{1}(t),t) \ge c_{1}t , \quad \Phi_{r}(\bar{x}_{1}(t),t) \ge r_{1}
			\end{split}
		\end{equation} and \begin{equation}\label{eq:x2}
		\begin{split}
				\Phi_{\theta}(\bar{x}_{2}(t),t)\le c_{0}t, \quad \Phi_{r}(\bar{x}_{2}(t),t) \ge \frac{N}{2} 
		\end{split}
	\end{equation} respectively, where $c_{0},c_{1},r_{1}>0$ are some universal numbers satisfying $c_{1}>c_{0}.$
	\end{proposition}

	\begin{proof}[Proof of Theorem \ref{thm:patch} from Proposition \ref{prop:patch-key}]
		We shall use the following geometric lemma: \begin{lemma}\label{lem:geom}
			Let $\gmm:[0,1] \to \mathbb{R}^{2}_+ $ be a simple smooth curve in $\mathbb{R}^{2}_+ := \left\{ (\varphi,r) : r > 0 \right\}$ equipped with the metric $ds^2 := r^2(dr^2+d\varphi^2)$ satisfying \begin{itemize}
				\item $\gmm(0) = (0,r_0)$, $\gmm(1) = (2\pi M, r_M)$ for some $r_{0},r_{M}>0$ and integer $M\ge1$;
				\item the image of $\gmm$ does not intersect with its $2\pi$ translates in $\varphi$; $\gmm([0,1]) \cap ( \gmm([0,1]) + (2\pi k, 0) ) = \emptyset$ for any integer $k$. 
			\end{itemize} Then we have $\mathrm{length}(\gmm([0,1])) \ge 2M \min\{ r_{0}, r_{M} \}  $. 
		\end{lemma}
	
    This result is nontrivial since the curve may spend most of its time near the origin; Figure \ref{fig:geodesic} gives an idea about the shortest curve satisfying the assumptions of the lemma. 
 Returning to $\mathbb{R}^2\backslash\{ 0\},$ this lemma says that a non self-intersecting curve in $\mathbb{R}^2\backslash\{ 0\}$ which winds around the origin $M$ times and start and end outside the unit disc has length at least $2M$. 
		
		\begin{figure}
			\centering
			\includegraphics[scale=1]{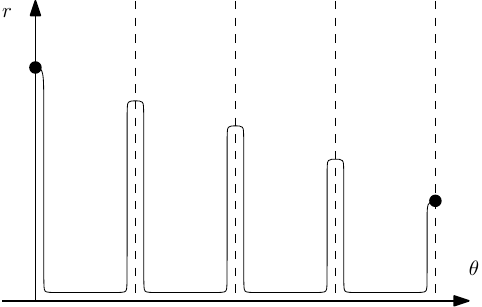}  
			\caption{Illustration for Lemma \ref{lem:geom}} \label{fig:geodesic}
		\end{figure}

		\begin{proof}[Proof of Lemma \ref{lem:geom}]
			Without loss of generality, we assume that $1 \le r_0 < r_M$. Furthermore, we may assume that $r_{0}$ and $r_{M}$ is the maximal (in $r$) intersection point of $\gmm([0,1])$ with $\{ \varphi = 0\}$ and  $\{ \varphi = 2\pi M \}$, respectively, since otherwise it means that we can restrict to a subset of $\gmm([0,1])$ which has all the desired properties. The goal is to prove the following 
			
			\medskip 
			
			\noindent \textbf{Claim}. For any integer $0< M' \le M$, the image $\gmm([0,1])$ intersects a point of the form $(2\pi M', r_{M'})$ with $r_{M'} \ge r_{0}$. (For simplicity, let us always denote $r_{M'}>0$ by the maximal (in $r$) intersection point of $\gmm([0,1])$ with $\{ \varphi = 2\pi M'\}$, which is well-defined by the intermediate value theorem.)
			
			\medskip
			
			\noindent Assume that this is the case. Then, the length of the image $\gmm([0,1])$ restricted to the region $\left\{ 2\pi (M'-1) \le \varphi \le 2\pi M' \right\}$ cannot be less than the distance (defined by the above metric on $\mathbb{R}^{2}_+$) between the points $(2\pi M', r_{M'})$ and $(2\pi (M'-1), r_{M'-1})$, which is simply equal to $2$. Since this holds for all $M' = 1, \cdots, M$, we obtain the claimed lower bound. 
            \\
			
			We now proceed to prove the \textbf{Claim} by induction in $M$. When $M=1$, there is nothing to prove. Let us take some $M>1$, assuming that the \textbf{Claim} holds for all integers less than $M$. If there exists some $0<k<M$ such that $r_{k} \ge r_{0}$, we may split the curve $\gmm$ into two parts, the one connecting $(0,r_{0})$ from $(2\pi k, r_{k})$ and the other $(2\pi k, r_{k})$ from $(2\pi M , r_{M})$. Each of these curves satisfy the assumptions of the lemma (after translating in $\varphi$ if necessary), and from the induction hypothesis, we have that $\mathrm{length}(\gmm([0,1])) \ge 2k + 2(M-k) = 2M$. 
			
			Therefore, towards a contradiction we may assume that $r_{1}, \cdots, r_{M-1} < r_{0}$ and $\gmm([0,1])<2M$. From the maximality assumption of these points, $\gmm([0,1])$ does not intersect $\{ \varphi = 2\pi k, r > r_0 \}$ for $k = 0, 1, \cdots, M-1$. We now try to extend the curve $\gmm$ by adding to its end the segment $\{ r_{0} < r < r_{M}, \varphi = 2\pi M \}$, while keeping the condition that the extended curve should not intersect with its $2\pi$-translates in $\varphi$. 
			\begin{itemize}
				\item Case 1:			If this is possible, then by projecting the curve onto $\mathbb{R}^{2}\backslash \{ (0,0)\}$ (defined by $x=r\cos\varphi, y = r\sin\varphi$) gives a simple, piecewise smooth, closed curve whose winding number around the origin is $M$. This is a contradiction, thanks to the well-known topological statement that \textbf{any simple piecewise smooth closed curve in $\mathbb{R}^{2}\backslash \{ (0,0)\}$ has winding number $0, -1,$ or $1$.} (While this is a deep statement in general, it can be proved by elementary means with the piecewise smooth assumption. For instance, one can observe that it is sufficient to prove the statement for polygonal curves and then proceed by an induction argument on the number of vertices.)
				\item Case 2: If this is \textit{not} possible, $\gmm([0,1])$ should pass through a point of the form $(2\pi k, r)$ for some integer $k \in (-\infty, -1] \cup [M, \infty)$ and $r_{0}<r<r_{M}$. We now consider the following sub-cases: \begin{itemize}
					\item Case 2-1: $k=M$. In this case we can take the curve $\tilde{\gmm}$ which connects $\gmm(0)$ to $(2\pi M, r)$, which is simple, smooth, contained in $\gmm([0,1])$ and satisfies all of the assumptions of the lemma. Then we repeat the same argument with this strictly shorter curve. The Case 2-1 cannot occur indefinitely due to the smoothness assumption of $\gmm$. Therefore, after a finite number of steps, one should arrive at either Case 1 (which then completes the induction argument) or Case 2-2 or 2-3, see below.
					\item Case 2-2: $k>M$. This time, we can take the curve $\tilde{\gmm}$ which connects $\gmm(0)$ to $(2\pi k, r)$, which is simple, smooth, contained in $\gmm([0,1])$ and satisfies almost all of the assumptions of the lemma. Here, the point is that it only suffices to show that $\mathrm{length}(\tilde{\gmm}) \ge 2(M-1)$, since the part of $\gmm([0,1])$ not contained in $\tilde{\gmm}$ has length at least 2. Then, we try to obtain a contradiction by connecting  $(2\pi k, r)$ to $(2\pi k, r_{0})$. If this fails again, then first consider when it intersects a point of the form $(2\pi k, r')$ for some $r_0<r'<r$. As before, this case cannot happen infinitely many times. But then after some finite number of iterations we obtain the existence of some point $(2\pi k',r')$ for some $k'  \in (-\infty, -1] \cup [M, \infty)\backslash \{ k\}$. But then we can now further take another curve contained in $\tilde{\gmm}$ for which we only need to show that its length is $\ge 2(M-2)$. This procedure cannot be repeated indefinitely, which eventually leads one to Case 1. 
					\item Case 2-3: $k<0$. This case can be handled similarly to the Case 2-2. We omit the details. 
				\end{itemize}
			\end{itemize}
            Given the \textbf{Claim}, the proof of Lemma \ref{lem:geom} is now complete. 
	\end{proof}
	
		To apply the Lemma, we need the slightly more general case when $M$ is possibly a non-integer greater than $1$. Here, the statement is that $\mathrm{length}(\gmm([0,1])) \ge 2\lfloor M \rfloor$ where $\lfloor M \rfloor$ is the integer part of $M$. This can be proved similarly by (trying to) connect the endpoint $(M,r_{M})$ to $(\lfloor M \rfloor, r_{0})$. We omit the details. Then, from the assumptions of the proposition, we conclude that $|\partial\Omega(t)|\gtrsim t$, which finishes the proof. 
	\end{proof}



\subsection*{Proof of Proposition \ref{prop:patch-key}}\label{subsec:finish}

We begin by recalling the main theorem of Sideris and Vega \cite{SV}, which ensures a type of stability for the circular patch (see also the works \cite{WP,Tang,CL}): for all $t\ge0$, \begin{equation}\label{eq:SiVe}
		\begin{split}
			|\Omega(t) \triangle B_1|^{2} \le 4\pi \sup_{ \Omega_{0} \triangle B_1 } \left| |x|^{2} - 1 \right| \, |\Omega_{0} \triangle B_1|
		\end{split}
	\end{equation} where $|A\triangle B|$ denotes the measure of the symmetric difference of two sets $A$ and $B$. Applying \eqref{eq:SiVe} in our case, we obtain that $|\Omega(t) \triangle B_1|^{2} \le CN^{2}|\Omega_{0} \triangle B|$  and for any given $0<\ve\ll1$ and $N\gg1$, we can take $\dlt=\dlt(\ve,N)>0$ sufficiently small so that we obtain \begin{equation}\label{eq:SiVe}
		\begin{split}
			|\Omega(t) \triangle B_1| \le \ve
		\end{split}
	\end{equation} for all $t\ge0$. Let us now recall the following standard estimates.

 \begin{lemma}\label{VelocityDecay}
     Let $\omega\in L^1\cap L^\infty(\mathbb{R}^2)$. Then, if $u=\nabla^\perp\Delta^{-1}\omega,$ we have that
     \[\|u\|_{L^\infty}\leq 2(\|\omega\|_{L^1}\|\omega\|_{L^\infty})^{\frac{1}{2}}.\] If we have, moreover, that $\int |\omega(y)||y|^2 \, \rmd y<\infty,$ then
     \[|u(x)|\leq \frac{C}{|x|+1}(\|(1+|y|^2)\omega\|_{L^1}+\|\omega\|_{L^\infty}).\]  
     If, in addition, $\omega$ is $m$-fold symmetric for some $m\geq 3,$ we have 
     \[|u(x)|\leq \frac{C|x|}{(1+|x|)^2}(\|(1+|y|^2)\omega\|_{L^1}+\|\omega\|_{L^\infty}).\]
 \end{lemma}
 \begin{remark}
    In the third estimate, in order to ensure that $u(0)=0,$ we can more generally assume the vanishing of some global integrals. For this to be propagated in the Euler equation, it is best to assume that $\omega$ is $m$-fold symmetric for some $m\geq 2.$ When $m=2$, the same estimate holds with an extra logarithmic factor that would not hinder the subsequent proof. We have elected to take $m\geq 3$ for simplicity. 
\end{remark}

  \begin{proof} This is shown using the exact form of the Newtonian potential, which gives the bound
 \[|u(x)|\leq \frac{1}{2\pi}\int_{\mathbb{R}^2}\frac{|\omega(y)|}{|x-y|} \, \rmd y.\] For the first inequality, we simply break the integral on $\mathbb{R}^2$ into two integrals over the regions $B_{R}(x)$ and $B_{R}(x)^c$, where $R=\Big(\frac{\|\omega\|_{L^1}}{\|\omega\|_{L^\infty}}\Big)^{1/2}.$ For the second inequality, we do the same except now taking $R=\frac{1}{|x|+1}$. For the third inequality, we simply use the fact that $|u(x)|\leq C|x|\|\omega\|_{L^\infty},$ when $\omega$ is $m$-fold symmetric (as in \cite{EBounded}).
 \end{proof}

We  have the following stability estimates:
 \begin{corollary}
     Let $u_* = r\mu(r) e_\theta$ denote the velocity field associated to $\chi_{B_1}$ and $u$ the velocity field associated to $\chi_{\Omega(t)}$. Then, for all $t\in\mathbb{R}$, we have 
     \[\|u-u_*\|_{L^\infty}\leq C\ve^{1/2},\qquad \|u_{\theta}-\mu\|_{L^2}\leq C\ve^{1/4}.\] 
 \end{corollary}

\begin{proof}
    Using \eqref{eq:SiVe} together with the first bound in Lemma \ref{VelocityDecay}, the first inequality is immediate. To get the other inequality, we estimate the square integral of $u_\theta - \mu$ in regions (i) $|x| \le R$, (ii) $R<|x|\le L$, and (iii) $L < |x|$. In each of these regions, we use the bound on $\Vert |x|^{-1} (u-u_*) \Vert_{L^\infty} $, $\Vert u-u_* \Vert_{L^\infty}$, and the last inequality of Lemma \ref{VelocityDecay}, respectively. This gives $$ \|u_{\theta}-\mu\|_{L^2}^{2}\lesssim R^2 + \varepsilon \ln(L/R) + L^{-2}.$$ Choosing $R= \varepsilon^{1/2}$ and $L= \varepsilon^{-1/2}$ gives the stronger estimate $ \|u_{\theta}-\mu\|_{L^2} \lesssim \varepsilon^{1/2} \ln^{1/2}(1/\varepsilon). $
\end{proof}

\begin{proof}[Proof of Proposition \ref{prop:patch-key}]
 Following the recipe given in Section \ref{punctureddomains}, and in particular \eqref{twistingthmunbounded}, it follows that
 \begin{equation}\label{L4Stuff}\|\Phi_{\theta}-\theta-t \mu(\Phi_r)\|_{L^2}\leq C\ve^{1/8}|t|,\end{equation} for all $t\in\mathbb{R}.$   It is easy to deduce the existence of the $\bar{x}_1(t)$ and $\bar{x}_2(t)$ from Proposition \ref{prop:patch-key} lying on the outer boundary of the patch. The main tool in this is \eqref{L4Stuff} and the Jordan Curve Theorem. Indeed, the stability estimate \eqref{eq:SiVe} ensures that there is always an order 1 mass of particles from the bulk of the patch in $B_1(0).$ It thus follows from \eqref{L4Stuff}, after allowing $\ve$ to be sufficiently small, that there exist (a mass of) points from inside the patch for which $\Phi_{\theta}\geq\frac{1}{2}t$. In fact, it can be seen quite simply from Chebyshev's inequality that \textit{most} points inside the patch satisfy this.  Now, since the inside of the patch can never go outside, there must be a boundary point to the left of the bulk points. This establishes the existence of $\bar{x}_1(t)$. The existence of $\bar{x}_2(t)$ is similarly guaranteed since the ``hole" in the patch is taken with large area and there is thus nowhere to go but up. Indeed, it follows that for all time there is an order 1 mass of particles in the hole that is outside of $B_{N/2}(0)$. By \eqref{L4Stuff}, since $|\mu(r)|=\frac{1}{r^2}$ and again allowing $\ve$ to be sufficiently small, we see that there must exist an order 1 mass of particles from the hole with $\Phi_\theta\leq c_{0}t$ for some small $c_{0}>0.$ Since the particles inside the hole remain there for all time, we deduce the existence of $\bar{x}_2(t)$, concluding the proof.
\end{proof}

\section*{Acknowledgements}

We would like to thank A. Elgindi, B. Khesin, G. Misiolek, M. Marcinkowski and D. Sullivan for many useful discussions and comments. We especially thank K. Dembski and S. Tae for pointing out a few typos in the first version of the manuscript. We also thank the anonymous referees for many comments that improved the exposition. We would like to thank the Simons Center for Geometry and Physics for its hospitality during the workshop \textit{Singularity and Prediction in Fluids} during June 2022, where this work was initiated.
The research of TDD was partially supported by the NSF
DMS-2106233 grant,  NSF CAREER award \#2235395 as well as an Alfred P. Sloan Fellowship. The research of TME was partially supported by the NSF grant DMS-2043024, an Alfred P. Sloan Fellowship, as well as a Simons Fellowship. The work of IJ was supported by the Samsung Science and Technology Foundation under Project Number SSTF-BA2002-04.

\end{document}